\documentclass[11pt, a4paper]{amsart}

\usepackage{amsfonts, amsthm, amsmath, amscd, amssymb}
\usepackage[all]{xy}
\usepackage{hyperref}

\newcommand\m{\mathfrak{m}}
\newcommand\frakb{\mathfrak{b}}
\newcommand\fraku{\mathfrak{u}}

\newcommand\frakt{\mathfrak{t}}
\newcommand\frakg{\mathfrak{g}}
\newcommand\A{\mathcal{A}}
\newcommand\OO{\mathcal{O}}
\newcommand\CC{\mathbb{C}}

\newcommand\T{\mathcal{T}}
\newcommand\G{\mathcal{G}}
\newcommand\B{\mathcal{B}}
\newcommand\Bbar{\overline{\B}}
\newcommand\Gbar{\overline{\G}}
\newcommand\U{\mathcal{U}}

\newcommand\V{\mathcal{V}}

\newcommand\F{\mathcal{F}}
\newcommand\ZZ{\mathbb{Z}}
\newcommand\PP{\mathbb{P}}
\newcommand\FF{\mathbb{F}}
\newcommand\cbar{\overline{c}}
\DeclareMathOperator{\Spec}{Spec}
\DeclareMathOperator{\relSpec}{\bf{Spec}}
\DeclareMathOperator{\Pic}{Pic}

\DeclareMathOperator{\ad}{ad}
\DeclareMathOperator{\GL}{GL}
\DeclareMathOperator{\PGL}{PGL}
\DeclareMathOperator{\Aut}{Aut}
\DeclareMathOperator{\Isom}{Isom}
\DeclareMathOperator{\Ext}{Ext}
\DeclareMathOperator{\Hom}{Hom}
\DeclareMathOperator{\End}{End}
\DeclareMathOperator{\Char}{char}
\DeclareMathOperator{\Frob}{Frob}

\newcommand{\tS}{{\widetilde S}}
\newcommand{\Ssing}{S^{\mathrm{sing}}}
\newcommand{\OOhat}{{\widehat \OO}}
\renewcommand{\le}{\leqslant}
\renewcommand{\ge}{\geqslant}
\newcommand\GG{\mathbb{G}}
\newcommand\GaR{\GG_{\mathrm{a},R}}
\newcommand\Gm{\GG_\mathrm{m}}

\newcommand\fbar{\overline{f}}
\newcommand\II{\mathcal{I}}
\newcommand\Zred{Z^{\mathrm{red}}}
\newcommand\id{\mathrm{id}}
\newcommand\pr{\mathrm{pr}}

\newtheorem{theorem}{Theorem}[section]
\newtheorem*{thm}{Theorem}
\newtheorem{lemma}[theorem]{Lemma}
\newtheorem{proposition}[theorem]{Proposition}
\newtheorem{cor}[theorem]{Corollary}
\theoremstyle{definition}
\newtheorem{remark}[theorem]{Remark}

\numberwithin{equation}{section}

\begin{document}

\title[Degeneration of torsors over families of del Pezzo surfaces]
{Degeneration of torsors\\ over families of del Pezzo surfaces}

\author{Ulrich Derenthal}

\address{Institut f\"ur Algebra, Zahlentheorie und Diskrete
  Mathematik, Leibniz Universit\"at Hannover,
  Welfengarten 1, 30167 Hannover, Germany}
 
\email{derenthal@math.uni-hannover.de}

\author{Norbert Hoffmann}

\address{Department of Mathematics and Computer Studies, Mary Immaculate College, South Circular Road, 
Limerick, Ireland}

\email{norbert.hoffmann@mic.ul.ie}

\date{September 17, 2020}


\begin{abstract}
  Let $S$ be a split family of del Pezzo surfaces over a discrete
  valuation ring such that the general fiber is smooth and the special
  fiber has $\mathbf{ADE}$-singularities. Let $G$ be the reductive
  group given by the root system of these singularities.  We construct
  a $G$-torsor over $S$ whose restriction to the generic fiber is the
  extension of structure group of the universal torsor.  This extends
  a construction of Friedman and Morgan for individual singular del
  Pezzo surfaces. In case of very good residue characteristic, this torsor
  is unique and infinitesimally rigid.
\end{abstract}

\maketitle

\section{Introduction}

Cubic surfaces over $\CC$ have been studied since the 19th century by
Cayley, Clebsch, Schl\"afli, Segre, Manin, and many others. In
particular, the $27$ lines on smooth cubic surfaces have an
interesting combinatorial structure: their symmetry group is the Weyl
group of type $\mathbf E_6$. Schl\"afli classified cubic
surfaces with $\mathbf{ADE}$-singularities: the worst has type $\mathbf E_6$.
In this paper, we explore a geometric connection between smooth cubic
surfaces, singular cubic surfaces, and an algebraic group of type
$\mathbf E_6$.

More generally, the lines on a split smooth del
Pezzo surface $S$ of degree $d=1,2,3,4,5,6,7$ have a symmetry group that
is a Weyl group of type $\mathbf{E}_8, \mathbf{E}_7, \mathbf{E}_6,
\mathbf{D}_5, \mathbf{A}_4, \mathbf{A}_2+\mathbf{A}_1, \mathbf{A}_1$,
respectively. The underlying root system
\begin{equation*}
  \Phi_0 = \{\alpha \in \Lambda \mid (\alpha,\alpha)=-2,\ (\alpha,-K_S)=0\}
\end{equation*}
is the set of $(-2)$-classes in the Picard group $\Lambda = \Pic(S)$, where
$-K_S$ is the anticanonical class \cite{MR833513}.

If $S$ is a del Pezzo surface with $\mathbf{ADE}$-singularities, then
its minimal desingularization $\tS$ is a weak del Pezzo surface
\cite{MR579026}, \cite[\S 8]{MR2964027}. If $\tS$ is split of degree
$d \le 7$, then the $(-2)$-classes in $\Lambda = \Pic(\tS)$ again form
a root system $\Phi_0$ of the same type as above. It contains a
subsystem $\Phi$ corresponding to the singularities of $S$, whose set
$\Delta$ of simple roots consists of the $(-2)$-curves on $\tS$.

These combinatorial data also correspond to algebraic groups.  Let
$T:=\Hom(\Lambda,\Gm)$ be the torus with character group $\Lambda$.
Let $G \subset G_0$ be the split reductive groups with maximal torus
$T$ and root systems $\Phi \subset \Phi_0$, respectively.  Let
$B \subset G$ be the Borel subgroup containing $T$ given by
$\Delta$. Choose a Borel subgroup $B_0 \subset G_0$ containing $B$.

A fundamental tool in the arithmetic study of weak del Pezzo surfaces
$\tS$ are the universal torsors introduced by Colliot-Th\'el\`ene and
Sansuc \cite{MR54:2657,MR0447246,MR0447250,MR89f:11082}. These are
certain $T$-torsors over $\tS$.
They have been used, for example, to
study the Hasse principle and weak approximation (e.g., in
\cite{MR870307,MR876222}) and the Manin conjecture (e.g., in
\cite{MR1909606,MR2874644}) for certain $\tS$.

However, universal torsors $\T$ over $\tS$ never descend to $S$. This
observation combined with physical considerations led Friedman and
Morgan \cite{MR1941576} to a geometric connection between singular del
Pezzo surfaces and algebraic groups: They show over $\CC$ that it is
possible to lift $\T$ (along the canonical projection
$B_0 \twoheadrightarrow T$) to a $B_0$-torsor over $\tS$ such that the
induced $G_0$-torsor descends to $S$
\cite[Theorem~3.1]{MR1941576}. Their construction is based on their
work, partly with Witten, on principal bundles over elliptic curves
\cite{MR1468319, MR1618343, arXiv:9811130, arXiv:math/0006174,
  arXiv:math/0108104}.

A different geometric connection between smooth del Pezzo surfaces and
algebraic groups was conjectured by Batyrev and proved in
\cite{popov_diplom, MR2332612, MR2368955, MR2976314}: Every universal
torsor $\T$ over a split smooth del Pezzo surface $S$ of degree
$\le 5$ has a natural $T$-equivariant embedding into the affine cone
over a flag variety $G_0/P_0$, where $P_0 \subset G_0$ is a certain
maximal parabolic subgroup.

Singular del Pezzo surfaces appear naturally as degenerations of
smooth del Pezzo surfaces. For modern accounts of such degenerations,
see Corti \cite{MR1426888} and Hacking--Keel--Tevelev
\cite{MR2534095}.
We consider flat families $S$ of split del Pezzo surfaces of arbitrary
degree over a discrete valuation ring $R$ with residue field $k$ such
that the generic fiber of $S$ is smooth and the special fiber of $S$
has at most $\mathbf{ADE}$-singularities. In
Section~\ref{sec:del_pezzo}, we describe the precise setup and discuss
the geometry in more detail. In particular, we have a
desingularization $\tS \to S$ that is minimal in the special fibers
and an isomorphism in the generic fibers.

Our main result, proved in Section~\ref{sec:groups}, provides a
geometric connection between smooth del Pezzo surfaces, singular del
Pezzo surfaces, and algebraic groups:

\begin{thm}
  Every universal torsor $\T$ over $\tS$ can be lifted to a $B$-torsor
  $\B$ over $\tS$ such that the induced $G$-torsor $\G$ descends to a
  $G$-torsor $\G'$ over $S$. If $k$ has very good characteristic for the
  root system $\Phi$, then $\B$, $\G$, and $\G'$ are all unique up to
  isomorphisms, and infinitesimally rigid.
\end{thm}

As $G \subset G_0$, our $G$-torsor $\G'$ induces a $G_0$-torsor
over $S$. Since every individual singular del Pezzo surface
over $\CC$ can be extended to such a degenerating family of del Pezzo
surfaces over $R = \CC[[t]]$, our result extends that of Friedman and
Morgan. On the other hand, we see no direct relation to the work
\cite{popov_diplom, MR2332612, MR2368955, MR2976314} on Batyrev's conjecture.

We view the uniqueness as evidence that the $G$-torsor $\G'$ is naturally
associated with the family $S$. See (\ref{eq:good_primes}) for the notion of
very good characteristic, and Proposition \ref{prop:uniqueness} for the precise
uniqueness statement. The torsor $\G'$ is called \emph{infinitesimally rigid}
if $H^1(S,\ad(\G'))=0$ for its adjoint vector bundle $\ad(\G')$, and similarly
for the other torsors.  In Section~\ref{sec:d4}, we give an example of a
family of cubic surfaces with a $\mathbf{D}_4$ singularity over a residue
field $k$ of characteristic $2$ for which $\G'$ is not infinitesimally rigid.

The work of Friedman and Morgan is generalized to other rational
surfaces over $\CC$ with $\mathbf{ADE}$-singularities in
\cite{MR3244919}.  For a physically motivated related construction
over families of such surfaces with an emphasis on the case
$\mathbf{A}_n$, using a Fourier-Mukai transform, see \cite{arXiv:1510.05025}.

\subsection*{Acknowledgments}

We thank Yuri Tschinkel for introducing us to these questions. The
first author was supported by grant DE 1646/3-1 of the Deutsche
Forschungsgemeinschaft. The second author was supported by Mary
Immaculate College Limerick through the PLOA sabbatical program, and
by the Riemann Center for Geometry and Physics of Leibniz
Universit\"at Hannover.

\section{Degenerating del Pezzo surfaces}\label{sec:del_pezzo}

Let $R$ be a discrete valuation ring with quotient field $K$, maximal
ideal $\m \subset R$ and residue field $k = R/\m$. Recall that every
split smooth del Pezzo surface has degree $d \in \{1, \dots, 9\}$, and
is
\begin{enumerate}
\item[(i)] either a blow-up of $\PP^2$ in $9-d$ points $x_1, \dots,
  x_{9-d} \in \PP^2(K)$ in \emph{general position} (i.e., no three on a line,
  no six on a conic, and no eight on a cubic with one of them on a
  singularity),
\item[(ii)] or $\PP^1\times\PP^1$ for $d=8$.
\end{enumerate}

As a degeneration of case (i), we consider the chain of blow-ups
\begin{equation*}
  \tS = \tS_{9-d} \xrightarrow{p_{9-d}} \tS_{8-d} 
  \xrightarrow{\phantom{p_2}} \dots \xrightarrow{\phantom{p_2}}
  \tS_2 \xrightarrow{p_2} \tS_1 \xrightarrow{p_1} \tS_0 = \PP^2_R
\end{equation*}
where $p_i : \tS_i \to \tS_{i-1}$ is the blow-up in the closure $\bar
x_i \in \tS_{i-1}(R)$ of the preimage of $x_i$ in $\tS_{i-1}(K)$.  The
generic fiber $\tS_K$ is the blow-up of $\PP^2_K$ in $x_1, \dots,
x_{9-d}$, and therefore a del Pezzo surface of degree $d$ over $K$.

Here, we assume that the images of $\bar x_i$ in $\tS_{i-1}(k)$ are in
\emph{almost general position}, by which we mean that the image of
$\bar x_i$ does not lie on a $(-2)$-curve in $\tS_{i-1,k}$. 

As degenerations of case (ii), we consider $\PP^1$-bundles
\begin{equation*}
  \tS \to \PP^1_R
\end{equation*}
whose restriction to the generic fiber $\PP^1_K$ is the trivial bundle
\begin{equation*}
  \PP^1_K \times_K \PP^1_K = \PP(\OO_{\PP^1_K} \oplus \OO_{\PP^1_K}) \to \PP^1_K
\end{equation*}
and whose
restriction to the special fiber is either trivial or the Hirzebruch
surface
\begin{equation*}
  \FF_2 = \PP(\OO_{\PP^1_k}(-1) \oplus \OO_{\PP^1_k}(1)) \to \PP^1_k.
\end{equation*}

In both cases, the special fiber $\tS_k$ is a \emph{weak del Pezzo
surface} over $k$ \cite{MR579026}, \cite[\S 8]{MR2964027}, i.e., a
smooth rational surface whose anticanonical class is nef and big. In
fact, every split weak del Pezzo surface appears as such a blow-up of
$\PP^2$ or such a $\PP^1$-bundle over $\PP^1$.

\begin{lemma}\label{lem:canonical_sheaf}
  The canonical bundle $\omega_{\tS_k}$ of the special fiber $\tS_k$
  is isomorphic to the restriction of the canonical bundle $\omega_\tS$
  of the total space $\tS$. 
\end{lemma}

\begin{proof}
  The two differ by the normal bundle of $\tS_k$ in $\tS$, which is
  the pullback of the normal bundle of $\Spec( k)$ in $\Spec( R)$,
  and therefore trivial.
\end{proof}

\begin{lemma}\label{lem:base_change}
  The $R$-module $H^0(\tS, \omega_\tS^{-m})$ is free, and the natural map   
  \begin{equation*}
    H^0(\tS, \omega_\tS^{-m}) \otimes_R k \to H^0(\tS_k, \omega_{\tS_k}^{-m})
  \end{equation*}
  is an isomorphism, for each integer $m \ge 0$.
\end{lemma}

\begin{proof}
  We carry some arguments from \cite[\S III.3]{MR1440180} over
  to the weak del Pezzo surface $\tS_k$. We have
  \begin{equation*}
    H^1(\tS_k,\OO_{\tS_k}) = 0
  \end{equation*}
  since this is a birational invariant \cite[Proposition~V.3.4]{MR0463157}.
  Let $D$ be a general member of the anticanonical linear system on $\tS_k$.
  Then $D$ does not contain any $(-2)$-curve on $\tS_k$, since
  $\omega_{\tS_k}^{-1}$ is globally generated. Therefore,
  \begin{equation*}
    H^0(D, \omega_{\tS_k}^m \otimes \OO_D) = 0
  \end{equation*}
  for $m \ge 1$. Being a local complete intersection, $D$ has dualizing sheaf
  \begin{equation*}
    \omega_D = \det( I_{D \subset \tS_k}/I_{D \subset \tS_k}^2)^{\vee} \otimes \omega_{\tS_k}
    \cong \OO_{\tS_k}(-D)^{\vee}_{|D} \otimes \OO_{\tS_k}(-D) = \OO_D
  \end{equation*}
  according to \cite[Definition~6.4.7]{MR1917232}. Therefore, Serre duality on $D$ implies
  \begin{equation*}
    H^1(D, \omega_{\tS_k}^{-m} \otimes \OO_D) = 0
  \end{equation*}
  for $m \ge 1$. By means of the exact sequence
  \begin{equation*}
    H^1(\tS_k, \omega_{\tS_k}^{-(m-1)}) \to H^1(\tS_k, \omega_{\tS_k}^{-m})
      \to H^1(D, \omega_{\tS_k}^{-m} \otimes \OO_D) = 0,
  \end{equation*}
  and induction over $m$, we conclude that
  \begin{equation*}
    H^1(\tS_k, \omega_{\tS_k}^{-m}) = 0
  \end{equation*}
  for $m \ge 0$. Using Cohomology and Base Change \cite[Theorem~III.12.11]{MR0463157} 
  together with Lemma~\ref{lem:canonical_sheaf}, the claim follows.
\end{proof}

Choosing a sufficiently large integer $m$ and a basis of $H^0(\tS, \omega_\tS^{-m})$,
we get an anticanonical map
\begin{equation*}
  \phi: \tS \twoheadrightarrow S \subset \PP^N_R.
\end{equation*}
Up to isomorphism over $R$, the scheme $S$ does not depend in the choices made.
As $S$ is integral and $R$ is a discrete valuation ring, $S$ is flat over $R$
by \cite[Proposition~III.9.7]{MR0463157}. Lemma~\ref{lem:base_change}
implies that the special fiber $S_k$ of $S$ is the anticanonical image of
the weak del Pezzo surface $\tS_k$.

In particular, $S_k$ is a del Pezzo surface with at most
$\mathbf{ADE}$-singularities, and $\phi$ contracts precisely the
$(-2)$-curves on $\tS_k$.

\begin{proposition}\label{prop:O_tS_O_S}
  We have $\phi_* \OO_{\tS} = \OO_S$, and $R^i\phi_* \OO_\tS = 0$ for all $i > 0$.
\end{proposition}

\begin{proof}
  Since $R^i\phi_*$ commutes with flat base change, and the completion
  of $R$ is flat over $R$, we may assume without loss of generality
  that $R$ is complete.

  We show by induction that $\phi_*\OO_{n\tS_k} = \OO_{nS_k}$ and
  $R^i\phi_*\OO_{n\tS_k} = 0$ for all $i > 0$. For $n=1$, this holds
  by \cite[Th\'eor\`eme~V.2]{MR579026}.  The induction step follows
  from the short exact sequence
  \begin{equation*}
    0 \to \OO_{(n-1)\tS_k} \xrightarrow{\cdot \pi} \OO_{n\tS_k} \to \OO_{\tS_k} \to 0,
  \end{equation*}
  where $\pi \in R$ is a generator of $\m$, and its analog for $S_k$.
  
  Using the Theorem on Formal Functions
  \cite[Theorem~III.11.1]{MR0463157}, the claim follows.
\end{proof}

\begin{lemma}\label{lem:root_sequence}
  Let $\Phi$ be a simply laced, irreducible root system, with simple roots
  $\Delta$ and positive roots $\Phi^+$. Let $\beta,\gamma \in \Phi^+$ such
  that $\gamma-\beta = \sum_{i=1}^t \alpha_i$, with all $\alpha_i \in \Delta$.
  Then there exist positive roots
  $\beta = \beta_0, \beta_1, \dots, \beta_t = \gamma$ such that
  $\beta_{i+1}-\beta_i \in \Delta$ for all $i$.
\end{lemma}

\begin{proof}
  We argue by induction on $t$; the cases $t=0$ and $t=1$ are clear. For $t
  \ge 2$, we note that
  \begin{equation*}
    -2 = (\gamma,\gamma) = (\gamma,\beta)+\sum_{i=1}^t (\gamma,\alpha_i).
  \end{equation*}
  The roots $\beta$ and $\gamma$ are not proportional since $\Phi$ is simply
  laced and since they are both positive, but not equal. Hence
  $(\gamma,\beta) \in \{0,1,-1\}$ \cite[Proposition
  IV.1.8]{MR0240238}. Therefore, $(\gamma,\alpha_i)<0$ for at least one
  $i$. Since both are positive, but not equal, we have
  $(\gamma,\alpha_i)=-1$. We define $\beta_{t-1}:=\gamma-\alpha_i$, which is a
  root since
  \begin{equation*}
    (\beta_{t-1},\beta_{t-1}) =
    (\gamma,\gamma)-2(\gamma,\alpha_i)+(\alpha_i,\alpha_i) = -2 + 2 -2 = -2,
  \end{equation*}
  and which is positive since it is the sum of $\beta$ and all $\alpha_j$ with
  $j \ne i$. By induction, we find a sequence
  $\beta = \beta_0, \dots, \beta_{t-1}$ as required.
\end{proof}

For the rest of this section, we fix one singular point $x$ on $S_k$.
Let $D_1, \dots, D_r$ be the $(-2)$-curves on $\tS_k$ that map to $x$. Let
\begin{equation*}
  Z=n_1D_1+\dots+n_rD_r
\end{equation*}
with $n_1, \dots, n_r \ge 1$ denote the \emph{fundamental cycle} on $\tS_k$ over $x$
(see \cite{MR0199191}). It has the property that $(Z,D_i) \le 0$ for all $i=1, \dots, r$,
and is minimal with this property. Here $(\cdot,\cdot)$ denotes the intersection number
of divisors on $\tS_k$. Put $N := n_1 + \dots + n_r$, and $\Zred:=D_1+\dots+D_r$.

\begin{lemma}\label{lem:combinatorial}
  There is a sequence of effective divisors
  \begin{equation}\label{eq:divisor_sequence}
    0=Z_0 < Z_1 < Z_2 < \dots < Z_r=\Zred < \dots < Z_N=Z
  \end{equation}
  on $\tS_k$ such that
  \begin{itemize}
  \item $Z_j-Z_{j-1}$ is a $(-2)$-curve $D_{i_j}$ for all $j=1, \dots, N$, and
  \item $(Z_j,D_i) \le 1$ for all $i=1, \dots, r$ and $j=1, \dots, N$.
  \end{itemize}
\end{lemma}

\begin{proof}
  The classes of $D_1, \dots, D_r$ are the simple roots of an
  irreducible root system $\Phi$ of type $\mathbf{A}_r$ ($r \ge 1$) or
  $\mathbf{D}_r$ ($r \ge 4$) or $\mathbf{E}_r$ ($r = 6, 7, 8$). By
  \cite[Remark~0.2.1]{MR986969}, the fundamental cycle $Z$ is its
  maximal root. The reduced fundamental cycle $\Zred$ is a positive
  root because
  \begin{equation*}
    (\Zred,\Zred) = \sum_{i=1}^r (D_i,D_i)+2\sum_{1 \le i < j \le r} (D_i,D_j) 
    = r\cdot (-2)+2 \cdot (r-1)\cdot 1 = -2
  \end{equation*}
  and $(\Zred,\omega_{\tS_k}^{-1}) = 0$.

  Applying Lemma~\ref{lem:root_sequence} first to $D_1$ and $\Zred$
  and then to $\Zred$ and $Z$ gives a sequence of effective divisors
  as in (\ref{eq:divisor_sequence}) such that $Z_j-Z_{j-1}$ is a
  simple root and therefore a $(-2)$-curve. Since $\Phi$ is simply
  laced, each of the positive roots $Z_j$ has intersection number $\le 1$
  with each of the simple roots $D_i$.
\end{proof}

\begin{lemma} \label{lem:Z_j}
  Let $Z_j \subset \tS_k$ be the closed subschemes given by Lemma~\ref{lem:combinatorial}.
  \begin{itemize}
   \item[(i)] $H^1(D_{i_j}, \II_{Z_{j-1} \subset Z_j}) = 0$ for $j=1 , \dots, N$.
    \smallskip
   \item[(ii)] $H^1(Z, \II_{Z \subset \tS_k}^n/\II_{Z \subset \tS_k}^{n+1}) = 0$ for $n \ge 0$.
  \end{itemize}
\end{lemma}

\begin{proof}
  The ideal sheaf of the effective divisor $Z$ on the smooth projective surface $\tS_k$
  is the line bundle $\OO( -Z) := \OO_{\tS_k}( -Z)$. Therefore, we have
  \begin{equation} \label{eq:I^n}
    \dfrac{\II_{Z \subset \tS_k}^n}{\II_{Z \subset \tS_k}^{n+1}} \cong \dfrac{\OO(-nZ)}{\OO(-(n+1)Z)}
      \cong \dfrac{\OO_{\tS_k}}{\II_{Z \subset \tS_k}} \otimes \OO(-nZ)
      \cong \OO_Z \otimes \OO(-nZ).
  \end{equation}
  Since $Z_j - Z_{j-1} = D_{i_j}$ according to Lemma~\ref{lem:combinatorial}, we similarly have
  \begin{equation} \label{eq:Z_j}
    \II_{Z_{j-1} \subset Z_j} \cong \dfrac{\OO(-Z_{j-1})}{\OO(-Z_j)}
      \cong \dfrac{\OO_{\tS_k}}{\II_{D_{i_j} \subset \tS_k}} \otimes \OO(-Z_{j-1})
      \cong \OO_{D_{i_j}} \otimes \OO(-Z_{j-1}),
  \end{equation}
  which is a line bundle of degree $(-Z_{j-1},D_{i_j}) \ge -1$ on $D_{i_j} \cong \PP^1$.
  But the first cohomology of any such line bundle vanishes. This proves part (i).

  Twisting the isomorphism \eqref{eq:Z_j} by the line bundle $\OO( -nZ)$ on $\tS_k$, we get
  \begin{equation*}
    \II_{Z_{j-1} \subset Z_j} \otimes \OO( -nZ) \cong \OO_{D_{i_j}} \otimes \OO(-Z_{j-1}-nZ),
  \end{equation*}
  which is now a line bundle of degree $(-Z_{j-1}-nZ,D_{i_j})$ on $D_{i_j} \cong \PP^1$.
  But this degree is still $\ge -1$, because the fundamental cycle $Z$ satisfies
  $(Z,D_{i_j}) \le 0$ by definition, and $n \ge 0$ by assumption. Hence we have more generally
  \begin{equation*}
    H^1( D_{i_j}, \II_{Z_{j-1} \subset Z_j} \otimes \OO( -nZ)) = 0.
  \end{equation*}
  Using induction over $j$, and the short exact sequences
  \begin{equation*}
    0 \to \II_{Z_{j-1} \subset Z_j} \to \OO_{Z_j} \to \OO_{Z_{j-1}} \to 0
  \end{equation*}
  twisted by the line bundle $\OO( -nZ)$ on $\tS_k$, we conclude that
  \begin{equation*}
    H^1( Z, \OO_Z \otimes \OO( -nZ)) = 0.
  \end{equation*}
  Because of the isomorphism \eqref{eq:I^n}, this proves part (ii) of the lemma.
\end{proof}

\begin{proposition} \label{prop:I^n/I^n+1}
  $H^1(Z, \II_{Z \subset \tS}^n/\II_{Z \subset \tS}^{n+1}) = 0$ for $n \ge 0$.
\end{proposition}

\begin{proof}
  Let $\pi \in R$ be a generator of $\m$. We first claim that the inclusion
  \begin{equation} \label{eq:pi_I^n}
    \pi \II_{Z \subset \tS}^n \subset \II_{Z \subset \tS}^{n+1} \cap \pi \OO_{\tS} 
  \end{equation}
  is an equality. It suffices to check this over the local ring $\OO_{\tS, z}$ of each point $z \in Z$.
  We choose a local function $f \in \OO_{\tS, z}$ whose residue class
  \begin{equation*}
    \fbar \in \OO_{\tS, z}/\pi\OO_{\tS, z} = \OO_{\tS_k, z}
  \end{equation*}
  is a local equation for the divisor $Z \subset \tS_k$.
  Then $\pi$ and $f$ generate $\II_{Z \subset \tS}$ in $z$.
  Hence $\pi \II_{Z \subset \tS}^n$ and $f^{n+1}$ generate $\II_{Z \subset \tS}^{n+1}$ in $z$.
  Suppose that
  \begin{equation*}
    f^{n+1} g \in \pi \OO_{\tS, z}
  \end{equation*}
  for some $g \in \OO_{\tS, z}$. Then its residue class $\overline{g} \in \OO_{\tS_k, z}$ satisfies
  \begin{equation*}
    \overline{f}^{n+1} \overline{g} = 0 \in \OO_{\tS_k, z}.
  \end{equation*}
  Since $\tS_k$ is integral and $\overline{f} \neq 0$, this implies $\overline{g} = 0$,
  and hence $g \in \pi \OO_{\tS, z}$. In particular, $f^{n+1} g$ lies in $\pi \II_{Z \subset \tS}^n$.
  Therefore, \eqref{eq:pi_I^n} is indeed an equality.

  Because of the natural short exact sequence
  \begin{equation*}
    0 \to \OO_{\tS} \xrightarrow{ \cdot \pi} \II_{Z \subset \tS} \to \II_{Z \subset \tS_k} \to 0,
  \end{equation*}
  the induced map $\II_{Z \subset \tS}^n/\II_{Z \subset \tS}^{n+1} \to
  \II_{Z \subset \tS_k}^n/\II_{Z \subset \tS_k}^{n+1}$ is surjective with kernel
  \begin{equation*}
    (\II_{Z \subset \tS}^n \cap \pi \OO_{\tS})/(\II_{Z \subset \tS}^{n+1} \cap \pi \OO_{\tS}).
  \end{equation*}
  As \eqref{eq:pi_I^n} is an equality,
  this kernel is $\pi \II_{Z \subset \tS}^{n-1}/\pi \II_{Z \subset \tS}^n$. Thus the sequence
  \begin{equation*}
    0 \to \II^{n-1}_{Z \subset \tS}/\II_{Z \subset \tS}^n
      \xrightarrow{ \cdot \pi} \II_{Z \subset \tS}^n/\II_{Z \subset \tS}^{n+1}
      \to \II_{Z \subset \tS_k}^n/\II_{Z \subset \tS_k}^{n+1} \to 0
  \end{equation*}
  is exact. The proposition follows from this by induction over $n$,
  using part (ii) of Lemma~\ref{lem:Z_j} for the case $n = 0$ and for the induction step.
\end{proof}

\section{Reductive groups and universal torsors}\label{sec:groups}

We continue in the setting of Section~\ref{sec:del_pezzo} and
construct certain algebraic groups naturally associated to the Picard
group of $\tS_k$.  

Since $\tS$, $\tS_K$ and $\tS_k$ are obtained by the same sequence of
blow-ups of a $\PP^2$, or all as $\PP^1$-bundles over $\PP^1$, the
canonical restriction maps
\begin{equation*}
  \Pic( \tS_K) \gets \Pic( \tS) \to \Pic( \tS_k)
\end{equation*}
are isomorphisms; we denote this abelian group by $\Lambda$.
The canonical bundles of $\tS$, $\tS_K$ and
$\tS_k$ define the same class in $\Lambda$ due to
Lemma~\ref{lem:canonical_sheaf}; we denote it by $K_{\tS} \in
\Lambda$.

The intersection forms on $\tS_K$ and on $\tS_k$ define the same
bilinear form $(\cdot,\cdot)$ on $\Lambda$. Let $\Lambda^\vee$ be the
dual of $\Lambda$, and denote the canonical pairing between
$\Lambda^\vee$ and $\Lambda$ by $\langle\cdot,\cdot\rangle$.  The root
system of the smooth del Pezzo surface $\tS_K$ is the set
\begin{equation}\label{eq:Psi}
  \Phi_0 := \{ \alpha \in \Lambda \mid (\alpha,\alpha)=-2,\ (\alpha,-K_{\tS})=0\}
\end{equation}
of $(-2)$-classes in $\Lambda$. It has type $\mathbf{E}_8,
\mathbf{E}_7, \mathbf{E}_6, \mathbf{D}_5, \mathbf{A}_4,
\mathbf{A}_2+\mathbf{A}_1, \mathbf{A}_1$ for $d=1,2,3,4,5,6,7$,
respectively, and type $\mathbf{A}_1$ for $d=8$ in the case (ii) of
$\PP^1$-bundles. Otherwise, $\Phi_0 = \emptyset$.

Let $\Phi \subset \Phi_0$ be the set of
$(-2)$-classes that are effective or anti-effective on $\tS_k$.  Put
\begin{equation*}
  \Phi^\vee:= \{\alpha^\vee \in \Lambda^\vee \mid \alpha \in \Phi\},
\end{equation*}
where $\alpha^\vee \in \Lambda^\vee$ is defined by
$\langle\alpha^\vee,x\rangle := (-\alpha,x)$.  Then a simple computation shows
that $(\Lambda,\Phi,\Lambda^\vee, \Phi^\vee)$ is a reduced root datum in the
sense of \cite[Expos\'e XXI, D\'efinition~1.1.1, 2.1.3]{MR2867622}. Let $G$ be
the associated split reductive group over $R$ \cite[Expos\'e XXV,
Corollaire~1.2]{MR2867622}.  Then the commutator subgroup $[G,G]$ is a
semisimple group over $R$ whose Dynkin diagram has the same type as the
singularities of $S_k$. The dimension of the maximal torus quotient $G/[G,G]$
is $10-d$ minus the rank of this Dynkin diagram.

Let $\frakg$ be the Lie algebra of $G$, with root spaces
$\frakg_\alpha \subset \frakg$ for $\alpha \in \Phi$.  The maximal torus $T$
of $G$ has character group $\Lambda$.  Therefore, $T$, $T_K$, and $T_k$ are
the N\'eron-Severi tori of $\tS$, $\tS_K$, and $\tS_k$, respectively.  Let $B$
be the Borel subgroup of $G$ containing $T$ such that the associated set
$\Delta$ of simple roots in $\Phi$ is the set of classes of the $(-2)$-curves
on $\tS_k$. Let $\frakt$ and $\frakb$ be the Lie algebras of $T$ and of $B$,
respectively.  The corresponding set $\Phi^+$ of positive roots consists
precisely of the effective $(-2)$-classes on $\tS_k$.

By a \emph{universal torsor} over $\tS$, we mean a $T$-torsor $\T$ such that
the $\Gm$-torsor $\lambda_*\T$ over $\tS$ obtained by extension of structure
group along every character $\lambda \in \Lambda = \Hom(T,\Gm)$ has class
$\lambda \in \Lambda = \Pic(\tS)$. Such universal
torsors exist and are unique up to isomorphism because $T \cong
\Gm^{10-d}$.

\begin{remark}
  The notion of universal torsor has been defined by Colliot-Th\'el\`ene and
  Sansuc \cite[(2.0.4)]{MR89f:11082} over base fields and, more generally, by
  Salberger \cite[Definition~5.14]{MR1679841} over Noetherian base
  schemes. Our definition is a special case of Salberger's.
\end{remark}

Let $\T$ be a universal torsor over $\tS$. Then the line bundle
$L_\alpha := \T \times^T \frakg_\alpha$ on $\tS$ has class
$\alpha \in \Lambda = \Pic(\tS)$ for each $\alpha \in \Phi$.

\begin{lemma}\label{lem:h1}
  For $\alpha \in \Phi^+$, the $R$-module $H^1(\tS,L_\alpha)$ is
  non-zero, cyclic and torsion (hence isomorphic to $R/\m^{n_\alpha}$
  for some $n_\alpha \ge 1$), the canonical map
  \begin{equation}\label{eq:base_change}
    H^1(\tS,L_\alpha) \otimes_R k \to H^1(\tS_k, L_{\alpha,k})
  \end{equation}
  is an isomorphism, and $H^0(\tS, L_\alpha) = H^2(\tS, L_\alpha)=0$.
\end{lemma}

\begin{proof}
  Since $\alpha$ is effective on $\tS_k$, we know that
  \begin{equation}\label{eq:Hi_Lalpha_over_k}
    \dim H^i(\tS_k, L_{\alpha,k}) = \begin{cases} 1 & \text{for } i=0,1,\\
      0 & \text{for } i=2.\end{cases}
  \end{equation}
  Indeed, $p_a(\tS_k)=p_a(\PP^2_k)=0$ since the arithmetic genus is a birational invariant,
  and hence the Riemann--Roch formula gives
  \begin{equation*}
    \chi(L_{\alpha,k})=0.
  \end{equation*}
  The class $K_{\tS} - \alpha$ has intersection number $-d < 0$ with the nef class $-K_{\tS}$,
  and is therefore not effective. Consequently, Serre duality gives
  \begin{equation*}
    H^2(\tS_k, L_{\alpha,k})=0.
  \end{equation*}
  Since the anticanonical morphism $\tS_k \to S_k$ is birational, there are only
  finitely many curves on $\tS_k$ whose intersection number with $-K_{\tS}$ is $0$.
  But every curve of class $\alpha$ has this property, which implies
  \begin{equation*}
    \dim H^0(\tS_k, L_{\alpha,k}) \le 1.
  \end{equation*}
  As $\alpha$ is effective on $\tS_k$, we get $H^0(\tS_k,L_{\alpha,k}) \cong k$,
  and hence also
  \begin{equation} \label{eq:H^1(L_alpha)}
    H^1(\tS_k,L_{\alpha,k}) \cong k.
  \end{equation}
  
  Over $K$ instead of $k$, the same arguments apply, but
  $\alpha$ is not effective over $\tS_K$, and therefore
  \begin{equation}\label{eq:Hi_over_K}
    H^i(\tS_K, L_{\alpha,K}) = 0
  \end{equation}
  for $i=0,1,2$.

  This implies that $H^1(\tS,L_\alpha)$ is torsion, and
  $H^2(\tS, L_\alpha)=0$ by Grauert's Theorem \cite[Corollary~III.12.9]{MR0463157}.
  Each section of the line bundle $L_\alpha$ vanishes on the generic fiber $\tS_K$, and hence on
  $\tS$. Therefore, $H^0(\tS,L_\alpha)=0$.

  Applying Cohomology and Base Change, we consider the natural maps
  \begin{equation*}
    \varphi^i : H^i(\tS, L_\alpha) \otimes_R k \to H^i(\tS_k,L_{\alpha,k}).
  \end{equation*}
  For $i=2$, both sides vanish. Using \cite[Theorem~III.12.11]{MR0463157} twice,
  we conclude first that $\varphi^1$ is surjective, and then that $\varphi^1$
  is an isomorphism. Due to \eqref{eq:H^1(L_alpha)}, this implies that
  $H^1(\tS,L_\alpha)$ is non-zero and cyclic.
\end{proof}

For $\alpha,\beta \in \Phi^+$ with $\alpha+\beta \in \Phi^+$, the Lie
bracket $[\_,\_] : \frakg_\alpha \otimes \frakg_\beta \to
\frakg_{\alpha+\beta}$ induces a morphism
\begin{equation}\label{eq:Lie_bracket}
  [\_,\_] : L_\alpha \otimes L_\beta \to L_{\alpha+\beta}.
\end{equation}

\begin{lemma}\label{lem:cup}
  Let $\alpha \in \Delta$ and $\beta \in \Phi^+$ such that
  $\alpha+\beta \in \Phi^+$. Then the cup product
  \begin{equation*}
    k \otimes_k k \cong H^0(\tS_k, L_{\alpha,k}) \otimes_k H^i(\tS_k, L_{\beta,k})
      \to H^i(\tS_k, L_{\alpha+\beta,k}) \cong k
  \end{equation*}
  induced by (\ref{eq:Lie_bracket}) is non-zero for $i = 0, 1$.
\end{lemma}

\begin{proof}
  Choose a non-zero section
  \begin{equation*}
    s \in H^0(\tS_k, L_{\alpha,k}).
  \end{equation*}
  Then $s$ vanishes precisely on a $(-2)$-curve $D \subset \tS_k$, and the sequence
  \begin{equation} \label{eq:s_cup}
    0 \to L_{\beta,k} \xrightarrow{[s,\_]} L_{\alpha+\beta,k} \to L_{\alpha+\beta|D} \to 0
  \end{equation}
  of coherent sheaves on $\tS_k$ is exact. The line bundle $L_{\alpha+\beta|D}$ on $D \cong \PP^1_k$
  has degree $(\alpha, \alpha+\beta) = -2+1=-1$ since $(\alpha,\beta)=1$; consequently,
  \begin{equation*}
    H^i(D, L_{\alpha+\beta|D}) \cong H^i(\PP^1_k,\OO_{\PP^1_k}(-1)) = 0
  \end{equation*}
  for $i = 0, 1$. In the long exact cohomology sequence resulting from \eqref{eq:s_cup},
  \begin{equation*}
    H^i(\tS_k, L_{\beta,k}) \xrightarrow{[s,\_]} H^i(\tS_k,L_{\alpha+\beta,k})
  \end{equation*}
  is therefore an isomorphism for $i = 0, 1$.
\end{proof}

The next step is to lift our universal torsor $\T$ to a
$B$-torsor $\B$ over $\tS$. We construct $\B$ as follows.

For $\alpha \in \Phi^+$, let $U_\alpha \cong \GaR$ be the associated root group in $B$.
Let $U_{\ge 2}$ be the subgroup of $B$ generated by all $U_\alpha$ with $\alpha \not\in \Delta$,
and put $B_{\le 1} := B/U_{\ge 2}$. We have the exact sequence
\begin{equation}\label{eq:first_seq}
  0 \to U_{=1} := \bigoplus_{\alpha \in \Delta} U_\alpha \to B_{\le 1} \to T \to 0.
\end{equation}
Here $T$ acts on $U_{=1}$ by conjugation. Associated to the $T$-torsor
$\T$ over $\tS$, we thus obtain a fibration over $\tS$ with fiber
$U_{=1}$. This group scheme over $\tS$ is by construction the underlying
additive group scheme of $\bigoplus_{\alpha \in \Delta} L_\alpha$.

We will first lift $\T$ to a $B_{\le 1}$-torsor $\B_{\le 1}$ over
$\tS$.  This is possible because \eqref{eq:first_seq} comes with a
splitting $T \to B_{\le 1}$, and the lifts $\B_{\le 1}$ are parameterized
by
\begin{equation*}
  H^1(\tS, \bigoplus_{\alpha \in \Delta} L_\alpha).
\end{equation*}
To make this precise, we consider the commutative diagram
\begin{equation*}
  \xymatrix{
    0 \ar[r] & \bigoplus_{\alpha \in \Delta} \frakg_\alpha \ar[r]\ar[d]_{\pr_\alpha} & \frakb_{\le 1} = \frakt \oplus \bigoplus_{\alpha \in \Delta} \frakg_\alpha \ar[r]\ar[d]_{\ad_\alpha \oplus \pr_\alpha} & \frakt \ar[r]\ar[d]_{\ad_\alpha} & 0\\
    0 \ar[r] & \frakg_\alpha \ar[r]_-{\big(
    \begin{smallmatrix}
      0\\\id
    \end{smallmatrix}
    \big)} & \End(\frakg_\alpha) \oplus \frakg_\alpha \ar[r]_-{(\id\ 0)} & \End(\frakg_\alpha) \ar[r] & 0
  }
\end{equation*}
of $B_{\le 1}$-modules, where the upper exact sequence consists of the
Lie algebras of (\ref{eq:first_seq}), and $\ad_\alpha$ sends $t \in
\frakt$ to $[t,\_] : \frakg_\alpha \to \frakg_\alpha$. Given one lift
$\B_{\le 1}$, we obtain an associated commutative diagram
\begin{equation}\label{eq:diagram_L_alpha_O_tS}
  \xymatrix{
    0 \ar[r] & \bigoplus_{\alpha \in \Delta} L_\alpha \ar[r]\ar[d] & \ad( \B_{\le 1}) \ar[r]\ar[d] & \Lambda^{\vee} \otimes_{\ZZ} \OO_\tS \ar[r]\ar[d] & 0\\
    0 \ar[r] & L_\alpha \ar[r] & \B_{\le 1} \times^{B_{\le 1}} (\End(\frakg_\alpha) \oplus \frakg_\alpha) \ar[r] & \OO_\tS \ar[r] & 0
    }
\end{equation}
of vector bundles over $\tS$. We denote the extension class of the
lower exact sequence by 
\begin{equation}\label{eq:c_alpha}
  c_\alpha \in H^1(\tS,L_\alpha).
\end{equation}
The classes $c_\alpha$ for $\alpha \in \Delta$ classify the lift
$\B_{\le 1}$ (see \cite[Proposition~3.1.ii]{MR3013030}, for example).

We choose a particular lift $\B_{\le 1}$ such that, for each $\alpha
\in \Delta$, the component $c_\alpha$ of the class of $\B_{\le 1}$
generates $H^1(\tS, L_\alpha)$ as an $R$-module. This is possible by
Lemma~\ref{lem:h1}.

\begin{lemma} \label{lem:B1_unique}
  Let $\B_{\le 1}'$ be another lift of $\T$ such that
  $H^1(\tS, L_\alpha)$ is also generated by the component $c_\alpha'$
  of the class of $\B_{\le 1}'$ for each $\alpha \in \Delta$. Then
  there is an automorphism $\sigma$ of $G$ with $\sigma_{|T} = \id_T$
  such that the extension of structure group of $\B_{\le 1}$ along
  $\sigma_{|B_{\le 1}}: B_{\le 1} \to B_{\le 1}$ is isomorphic to
  $\B_{\le 1}'$ as a lift of $\T$.
\end{lemma}

\begin{proof}
  Since $c_\alpha$ and $c_\alpha'$ both generate $H^1(\tS,L_\alpha)$,
  we have $c_\alpha' = \lambda_\alpha c_\alpha$ with
  $\lambda_\alpha \in R^\times$ for each $\alpha \in \Delta$.
  According to \cite[Expos\'e XXIII, Th\'eor\`eme~4.1]{MR2867622},
  there is a unique automorphism $\sigma$ of $G$ with
  $\sigma_{|T} = \id_T$ such that $\sigma$ acts on $\frakg_\alpha$
  as multiplication by $\lambda_\alpha$ for each $\alpha \in \Delta$.
  This implies
  \begin{equation*}
    (\sigma_{|B_{\le 1}})_* \B_{\le 1} \cong \B_{\le 1}'
  \end{equation*}
  as lifts of $\T$.
\end{proof}

\begin{remark}
  This automorphism $\sigma$ of $G$ can be described as follows.
  The action of $G$ on itself by conjugation descends to an action
  of $G^{\ad} := G/Z$ on $G$, where $Z \subset G$ is the
  (scheme-theoretic) center. The subgroup $T^{\ad} := T/Z$ of
  $G^{\ad}$ acts trivially on $T$. Since $\Delta$ is a basis of the
  lattice $\Hom(T^{\ad}, \Gm)$, there is a unique point
  $t \in T^{\ad}(R)$ such that $\alpha(t) = \lambda_\alpha$ for each
  $\alpha \in \Delta$. Conjugation by this point $t$ is the
  required automorphism $\sigma$ of $G$.
\end{remark}

\begin{lemma} \label{lem:lift_to_B}
  The $B_{\le 1}$-torsor $\B_{\le 1}$ can be lifted to a $B$-torsor $\B$ over $\tS$.
\end{lemma}

\begin{proof}
  Let $\Phi^+_{=n}$ (resp.\ $\Phi^+_{\ge n}$) be the set of all $\alpha \in \Phi^+$
  that are sums of precisely (resp.\ at least) $n$ not necessarily distinct simple
  roots. Generalizing the above notation, we let $U_{\ge n}$ be the subgroup of $B$ generated by all
  $U_\alpha$ with $\alpha \in \Phi^+_{\ge n}$, and put $B_{\le n} := B/U_{\ge n+1}$.
  We have the exact sequences
  \begin{equation}\label{eq:next_seq}
    0 \to U_{=n}:=\bigoplus_{\alpha \in \Phi^+_{=n}} U_\alpha \to B_{\le n} \to B_{\le n-1} \to 0,
  \end{equation}
  in which $B_{\le n} = B/U_{\ge n+1}$ acts on $U_{=n}$ by conjugation.
  Here $U_{\ge 1}/U_{\ge n+1}$ acts trivially, because $[U_{\ge 1},U_{\ge n}] \subset U_{\ge n+1}$.
  Therefore, the action descends to an action of $B/U_{\ge 1} = T$ on $U_{=n}$.
  Associated to the $T$-torsor $\T$ over $\tS$, we thus obtain a fibration over $\tS$, with fiber
  $U_{=n}$. This group scheme over $\tS$ is by construction the underlying
  additive group scheme of $\bigoplus_{\alpha \in \Phi^+_{=n}} L_\alpha$.

  Using induction, we assume that $\B_{\le 1}$ can be lifted
  to a $B_{\le n-1}$-torsor $\B_{\le n-1}$ for some $n \ge 2$.
  We try to lift $\B_{\le n-1}$ to a $B_{\le n}$-torsor $\B_{\le n}$
  along the exact sequence \eqref{eq:next_seq}. The obstruction against
  such a lift is an element in
  \begin{equation*}
    H^2(\tS, \bigoplus_{\alpha \in \Phi^+_{=n}} L_\alpha)
  \end{equation*}
  (see \cite[Proposition~3.1.i]{MR3013030}). This cohomology vanishes by Lemma~\ref{lem:h1}.

  For sufficiently large $n$, we have $B_{\le n} = B$, and $\B := \B_{\le n}$
  is the required lift of $\B_{\le 1}$.
\end{proof}

According to \cite[12.12]{MR1642713}, we can choose a nonzero element
$x_\alpha \in \frakg_\alpha$ for each $\alpha \in \Phi^+$ such that
\begin{equation}\label{eq:epsilon_relation}
[x_\alpha,x_\beta] = \epsilon_{\alpha,\beta} x_\gamma
\end{equation}
with $\epsilon_{\alpha,\beta} \in \{-1,1\}$ for all
$\alpha,\beta,\gamma \in \Phi^+$ with $\alpha+\beta=\gamma$.

\begin{lemma}\label{lem:e_alpha-f_alpha}
  There are classes $e_\alpha \in H^0(\tS_k, L_{\alpha,k})$ for
  $\alpha \in \Phi^+$ and $f_\alpha \in H^1(\tS_k,L_{\alpha,k})$ for
  $\alpha \in \Phi^+$ such that
  \begin{enumerate}
  \item[(i)] $[e_\alpha, e_\beta] = \epsilon_{\alpha,\beta} e_\gamma$
    for all $\alpha,\beta,\gamma \in \Phi^+$ with
    $\alpha+\beta=\gamma$,
  \item[(ii)] $[e_\alpha, f_\beta] = \epsilon_{\alpha,\beta} f_\gamma$
    for all $\alpha,\beta,\gamma \in \Phi^+$ with
    $\alpha+\beta=\gamma$, and
  \item[(iii)] $f_\alpha = c_{\alpha,k}$ is the restriction of the class $c_\alpha$ from
    (\ref{eq:c_alpha}) for all $\alpha \in \Delta$.
  \end{enumerate}
\end{lemma}

\begin{proof}
  We choose a point $p \in \tS(k)$ outside the $(-2)$-curves, and a point
  $\tau \in \T(k)$ above $p$. The trivialization $\tau$ of $\T$ above $p$ induces an
  isomorphism $L_{\alpha,p} \to \frakg_\alpha$ for each $\alpha \in
  \Phi^+$. We define $e_\alpha \in H^0(\tS_k, L_{\alpha,k})$ as the unique
  section whose value at $p$ maps to $x_\alpha$ under this isomorphism. Then
  (i) holds by construction because of (\ref{eq:epsilon_relation}).

  For each irreducible component of $\Phi$, consider its highest root
  $\delta$, and choose a nonzero $f_\delta \in H^1(\tS_k, L_{\delta,k})$. Let
  $\alpha$ be a positive root in the same component. Since the anticanonical
  morphism $\tS_k \to S_k$ is birational, there are only finitely many curves
  on $\tS_k$ whose intersection number with $-K_\tS$ is $0$. But every curve
  of class $\delta-\alpha$ has this property, which implies
  \begin{equation*}
    \dim \Hom(L_{\alpha,k}, L_{\delta,k}) \le 1.
  \end{equation*}
  On the other hand, the divisor class $\delta-\alpha$ contains a sum of
  $(-2)$-curves.  Hence there is a unique morphism
  $\phi_\alpha : L_{\alpha,k} \to L_{\delta,k}$ whose restriction to $p$
  sends $x_\alpha$ to $x_\delta$. The induced map
  \begin{equation*}
    H^1(\tS_k, L_{\alpha,k}) \to H^1(\tS_k, L_{\delta,k})
  \end{equation*}
  is bijective because of Lemma~\ref{lem:root_sequence} and
  Lemma~\ref{lem:cup}. We define
  \begin{equation*}
    f_\alpha \in H^1(\tS_k, L_{\alpha,k})
  \end{equation*}
  as the inverse image of $f_\delta$.

  Assume that $\alpha+\beta=\gamma$ with $\beta,\gamma \in \Phi^+$.  Then
  $\beta$ and $\gamma$ lie in the same irreducible component of $\Phi$ as
  $\alpha$ and $\delta$. The diagram
  \begin{equation*}
    \xymatrix{L_\beta \ar[rr]^{\epsilon^{-1}_{\alpha,\beta}[e_\alpha,\_]} \ar[rd]_{\phi_\beta}
      && L_\gamma \ar[dl]^{\phi_\gamma}\\
      & L_\delta }
  \end{equation*}
  commutes, as can be seen by evaluating at $p$ and using
  $\epsilon_{\alpha,\beta}^{-1}[x_\alpha,x_\beta] = x_\gamma$. Therefore,
  $\epsilon_{\alpha,\beta}^{-1}[e_\alpha,f_\beta] = f_\gamma$, which proves (ii).

  We have $c_{\alpha, k} = \lambda_\alpha f_\alpha$ with
  $\lambda_\alpha \in k^\times$ for each $\alpha \in \Delta$
  by construction. For $\alpha = \sum_i \alpha_i$ with all
  $\alpha_i \in \Delta$, we put
  $\lambda_\alpha := \prod_i \lambda_{\alpha_i} \in k^\times$.
  Replacing $e_\alpha$ by $\lambda_\alpha e_\alpha$ and
  $f_\alpha$ by $\lambda_\alpha f_\alpha$ for $\alpha \in \Phi^+$
  preserves (i) and (ii) and ensures (iii).
\end{proof}

Recall from \cite[I.4]{MR0268192} that a prime $p$ is \emph{good} for
an irreducible root system if $p$ does not occur as a coefficient of
the highest root. A good prime $p$ is \emph{very good} if moreover $p$
does not divide the determinant of the Cartan matrix. For the simply
laced root systems, the very good primes are as follows.
\begin{equation}\label{eq:good_primes}
  \text{$\mathbf{A}_r$: all $p \nmid r+1$,}\quad 
  \text{$\mathbf{D}_r$: all $p\ne 2$,}\quad
  \text{$\mathbf{E}_6, \mathbf{E}_7$: all $p \ne 2,3$,}\quad
  \text{$\mathbf{E}_8$: all $p \ne 2,3,5$.}
\end{equation}
The field $k$ has \emph{very good characteristic} for $\Phi$ if $\Char(k)$
is $0$ or a very good prime for every irreducible component of $\Phi$.

\begin{lemma} \label{lem:cup_epi}
  Given $\beta \in \Phi^+_{=n-1}$ and $\gamma \in \Phi^+_{=n}$, we consider the map
  \begin{equation}\label{eq:cup_c_alpha}
    H^0( \tS_k, L_{\beta, k}) \to H^1( \tS_k, L_{\gamma, k})
  \end{equation}
  given by $[\_, c_{\alpha,k}]$ if $\gamma - \beta$ is a simple root $\alpha$,
  and by $0$ otherwise. If $k$ has very good characteristic for $\Phi$, then the
  sum
  \begin{equation}\label{eq:cup_c}
    \bigoplus_{\beta \in \Phi^+_{=n-1}} H^0( \tS_k, L_{\beta, k}) \to
    \bigoplus_{\gamma \in \Phi^+_{=n}} H^1( \tS_k, L_{\gamma, k})
  \end{equation}
  of all these maps is surjective for $n \ge 2$.
\end{lemma}

\begin{proof}
  Choose $e_\alpha,f_\alpha$ as in Lemma~\ref{lem:e_alpha-f_alpha}.  Given
  $\beta \in \Phi^+_{=n-1}$ and $\gamma \in \Phi^+_{=n}$ with
  $\gamma-\beta=\alpha \in \Delta$, the map \eqref{eq:cup_c_alpha} is given by
  the $(1\times 1)$-matrix $(\epsilon_{\beta,\alpha})$ with respect
  to the bases $\{e_\beta\}$ and $\{f_\gamma\}$ by
  Lemma~\ref{lem:e_alpha-f_alpha}(ii).

  Therefore, the matrix of \eqref{eq:cup_c} with respect to the bases
  $\{e_\beta \mid \beta \in \Phi^+_{=n-1}\}$ and
  $\{f_\gamma \mid \gamma \in \Phi^+_{=n}\}$ has entries
  $\epsilon_{\beta,\alpha}$ whenever $\gamma-\beta=\alpha \in \Delta$, and $0$
  otherwise.
  If $\Phi$ is reducible, then these matrices are block
  diagonal. Computing the ranks of all these matrices for all possible
  irreducible root systems shows that the maps in question are
  surjective in very good characteristic.
\end{proof}

\begin{proposition}\label{prop:B_unique} 
  Assume that $k$ has very good characteristic for $\Phi$. Let the
  $B$-torsor $\B$ over $\tS$ be an arbitrary lift of $\B_{\le 1}$.
  \begin{itemize}
   \item[(i)] Up to isomorphism of torsors, the restriction
    $\B_k := \B_{|\tS_k}$ does not depend on the choice of $\B$.
   \item[(ii)] The adjoint vector bundle $\ad( \B_k) \to \tS_k$
    satisfies $H^1(\tS_k, \ad( \B_k) ) = 0$.
  \end{itemize}
\end{proposition}

\begin{proof}
  Considering the restricted $B_{\le n}$-torsor
  \begin{equation*}
    \B_{\le n, k} := \B_{\le n|\tS_k}
  \end{equation*}
  over $\tS_k$, we argue by induction over $n$.
 
  For the proof of (ii), the construction of $\B_{\le 1}$ provides us
  with an exact sequence
  \begin{equation*}
    0 \to \bigoplus_{\alpha \in \Delta} L_{\alpha, k} \to \ad( \B_{\le 1, k})
      \to \Lambda^{\vee} \otimes_{\ZZ} \OO_{\tS_k} \to 0
  \end{equation*}
  of vector bundles over $\tS_k$. Since $H^0( \tS_k, \OO_{\tS_k}) = k$ and $H^1( \tS_k, \OO_{\tS_k}) = 0$,
  the resulting long exact cohomology sequence reads
  \begin{equation} \label{eq:H^1(ad_B1)}
    \Lambda^{\vee} \otimes_{\ZZ} k
      \xrightarrow{\delta} \bigoplus_{\alpha \in \Delta} H^1( \tS_k, L_{\alpha, k})
      \to H^1( \tS_k, \ad( \B_{\le 1, k})) \to 0.
  \end{equation}
  The connecting homomorphism $\delta$ is the sum over $\alpha \in \Delta$ of the compositions
  \begin{equation*}
    \Lambda^{\vee} \otimes_{\ZZ} k \xrightarrow{\alpha \otimes \id_k} k
      \xrightarrow{\_ \cdot c_{\alpha, k}} H^1( \tS_k, L_{\alpha, k}).
  \end{equation*}
  This composition is surjective, because $c_{\alpha, k}$ generates $H^1( \tS_k, L_{\alpha, k})$
  for each $\alpha \in \Delta$ and
  \begin{equation*}
    (\alpha \otimes \id_k)_{\alpha \in \Delta} : \Lambda^{\vee} \otimes_{\ZZ} k
      \to \bigoplus_{\alpha \in \Delta} k
  \end{equation*}
  is surjective (since its restriction to the span of the coroots in $\Lambda^{\vee}$
  is given by the Cartan matrix, which is invertible in $k$ by assumption).
  Consequently,
  \begin{equation*}
    H^1( \tS_k, \ad( \B_{\le 1, k})) = 0
  \end{equation*}
  according to the exact sequence \eqref{eq:H^1(ad_B1)}.

  By induction, we may assume the same for $\B_{\le n-1, k}$.
  We again have an exact sequence
  \begin{equation}\label{eq:ad_sequence}
    0 \to \bigoplus_{\gamma \in \Phi^+_{=n}} L_{\gamma, k}
      \to \ad( \B_{\le n, k}) \to \ad( \B_{\le n-1, k}) \to 0
  \end{equation}
  of vector bundles over $\tS_k$. Using the resulting long exact cohomology sequence,
  it remains to prove that its connecting homomorphism
  \begin{equation}\label{eq:ad_sequence_connecting}
    H^0( \tS_k, \ad( \B_{\le n-1, k}))
      \to \bigoplus_{\gamma \in \Phi^+_{=n}} H^1( \tS_k, L_{\gamma, k})
  \end{equation}
  is surjective. We consider its restriction
  \begin{equation}\label{eq:delta}
    \delta : \bigoplus_{\beta \in \Phi^+_{=n-1}} H^0( \tS_k, L_{\beta, k}) \to
    \bigoplus_{\gamma \in \Phi^+_{=n}} H^1( \tS_k, L_{\gamma, k})
  \end{equation}
to the subbundle
  \begin{equation*}
    \bigoplus_{\beta \in \Phi^+_{=n-1}} L_{\beta, k} \subset \ad( \B_{\le n-1, k}).
  \end{equation*}
  Choose $\beta \in \Phi^+_{=n-1}$ and $\gamma \in \Phi^+_{=n}$ such
  that $\gamma-\beta = \alpha \in \Delta$. Then the component 
  \begin{equation}\label{eq:delta_component}
    H^0( \tS_k, L_{\beta, k}) \to H^1( \tS_k, L_{\gamma, k})
  \end{equation}
  of $\delta$ is the connecting homomorphism of
  the exact sequence
  \begin{equation}\label{eq:sequence_L_gamma_L_beta}
    0 \to L_\gamma \to \B_{\le 1} \times^{B_{\le 1}} (\frakg_\beta \oplus \frakg_\gamma) \to L_\beta \to 0
  \end{equation}
  of vector bundles over $\tS$ associated with the exact sequence
  \begin{equation*}
    0 \to \frakg_\gamma \to \frakg_\beta \oplus \frakg_\gamma \to \frakg_\beta \to 0
  \end{equation*}
  of $B_{\le 1}$-modules. Let $-[\_,\_] : \frakg_\beta \otimes
  \frakg_\alpha \to \frakg_\gamma$ be the linear map that sends
  $x_\beta \otimes x_\alpha$ to
  $-[x_\beta,x_\alpha]=[x_\alpha,x_\beta]$. Then the isomorphism
  \begin{equation*}
    \xymatrix{
      0 \ar[r] & \frakg_\beta \otimes \frakg_\alpha \ar[r]^-{\big(
    \begin{smallmatrix}
      0\\\id
    \end{smallmatrix}
    \big)}\ar[d]_{-[\_,\_]} & \frakg_\beta \oplus (\frakg_\beta \otimes \frakg_\alpha) \ar[r]^-{(\id\ 0)}\ar[d]_{
        \big(\begin{smallmatrix}
          \id & 0 \\ 0 & -[\_,\_]
        \end{smallmatrix}\big)
} & \frakg_\beta \ar[r]\ar[d]_{\id} & 0 \\
      0 \ar[r] & \frakg_\gamma \ar[r]_-{\big(
    \begin{smallmatrix}
      0\\\id
    \end{smallmatrix}
    \big)} & \frakg_\beta \oplus \frakg_\gamma \ar[r]_-{(\id\ 0)} & \frakg_\beta \ar[r] & 0
    }
  \end{equation*}
  of short exact sequences is $B_{\le 1}$-equivariant. Therefore, it
  induces the isomorphism
  \begin{equation*}
    \xymatrix{
      0 \ar[r] & L_\beta \otimes L_\alpha \ar[r]\ar[d]_{-[\_,\_]} & \B_{\le 1} \times^{B_{\le 1}} (\frakg_\beta \oplus (\frakg_\beta \otimes \frakg_\alpha)) \ar[r]\ar[d] & L_\beta \ar[r]\ar[d]_{\id} & 0\\
      0 \ar[r] & L_\gamma \ar[r] & \B_{\le 1} \times^{B_{\le 1}} (\frakg_\beta \oplus \frakg_\gamma) \ar[r] & L_\beta \ar[r] & 0
    }
  \end{equation*}
  from the second exact sequence in (\ref{eq:diagram_L_alpha_O_tS})
  tensored with $L_\beta$ to
  (\ref{eq:sequence_L_gamma_L_beta}). Comparing the classes of these
  exact sequences, we see that
  \begin{equation*}
    -[\_,\_] : L_\beta \otimes L_\alpha \to L_\gamma
  \end{equation*}
  sends $c_\alpha \in H^1(\tS, L_\alpha) = \Ext^1(L_\beta, L_\beta
  \otimes L_\alpha)$ to the class of
  (\ref{eq:sequence_L_gamma_L_beta}) in $\Ext^1(L_\beta,
  L_\gamma)$. This shows that the component (\ref{eq:delta_component})
  of $\delta$ in question is given by
  \begin{equation*}
    -[\_, c_{\alpha,k}] : H^0( \tS_k, L_{\beta, k}) \to H^1( \tS_k, L_{\gamma, k}).
  \end{equation*}
  Therefore, $\delta$ is surjective according to
  Lemma~\ref{lem:cup_epi}. Hence (\ref{eq:ad_sequence_connecting}) is
  also surjective. This proves (ii).

  For the proof of (i), we may assume by induction that $\B_{\le n-1, k}$ does
  not depend on the choice of $\B$.  Since
  \begin{equation*}
    [U_{\ge n-1}, U_{\ge 1}] \subset U_{\ge n},
  \end{equation*}
  the subgroup $U_{= n-1} \subset B_{\le n-1}$ is normal, and the conjugation action of $B_{\le n-1}$
  on $U_{= n-1}$ factors through the action of $T$. This implies
  \begin{equation} \label{eq:H^0_in_Aut}
    \Aut( \B_{\le n-1, k}) \supset \bigoplus_{\beta \in \Phi^+_{=n-1}} H^0( \tS_k, L_{\beta, k}).
  \end{equation}
  The set of lifts of $\B_{\le n-1, k}$ to a $B_{\le n}$-torsor is a torsor under the group
  \begin{equation}\label{eq:lifts_torsor}
    \bigoplus_{\gamma \in \Phi^+_{=n}} H^1( \tS_k, L_{\gamma, k}).
  \end{equation}
  This set comes with an action of $\Aut( \B_{\le n-1, k})$, whose
  restriction to the subgroup in \eqref{eq:H^0_in_Aut} is the
  homomorphism $\delta$ in (\ref{eq:delta}). As we have seen, $\delta$
  is surjective by Lemma~\ref{lem:cup_epi}.  Hence $\Aut( \B_{\le n-1,
    k})$ acts transitively on the set of lifts $\B_{\le n, k}$.  Thus
  $\B_{\le n, k}$ does not depend on the choice of $\B$.

  Since $B_{\le n} = B$ for sufficiently large $n$, this proves part
  (i).
\end{proof}

\begin{remark}
  If $\Phi$ has type $\mathbf{D}_4$ and $k$ has characteristic $2$,
  then Lemma~\ref{lem:cup_epi} is not true, and hence our proof of
  Proposition~\ref{prop:B_unique} does not work in this case.

  In Section~\ref{sec:d4}, we will give an example of a family $S$ of
  cubic surfaces over a discrete valuation ring $R$ with residue field
  $k$ of characteristic $2$ that has a $\mathbf{D}_4$ singularity in the
  special fiber and for which Proposition~\ref{prop:B_unique} (ii) is false.
\end{remark}

\begin{lemma} \label{lem:j>>0}
  The canonical group homomorphism
  \begin{equation*}
    \Aut( \B) \to \Aut( \T) = T( R) \cong (R^{\times})^{10-d}
  \end{equation*}
  is injective, and its image contains the subgroup
  \begin{equation*}
    (1 + \m^j)^{10-d} \subset (R^{\times})^{10-d}
  \end{equation*}
  for any sufficiently large integer $j$.
\end{lemma}

\begin{proof}
  By construction of the torsors $\B_{\le n}$, we have exact sequences
  \begin{equation} \label{eq:Aut(B_n)}
    0 \to \bigoplus_{\alpha \in \Phi^+_{=n}} H^0( \tS, L_{\alpha})
      \to \Aut( \B_{\le n}) \to \Aut( \B_{\le n-1})
  \end{equation}
  for $n \ge 1$, with $\B_{\le 0} := \T$. Here $H^0( \tS, L_{\alpha}) = 0$ according to Lemma~\ref{lem:h1}.
  Therefore, the canonical group homomorphisms
  \begin{equation*}
    \Aut( \B) = \Aut( \B_{\le N}) \to \Aut( \B_{\le N-1}) \to \dots \to \Aut( \B_{\le 0}) = \Aut( \T)
  \end{equation*}
  are all injective. The obstruction against lifting an automorphism of $\B_{\le n-1}$ to an automorphism
  of $\B_{\le n}$ is an element in
  \begin{equation*}
    \bigoplus_{\alpha \in \Phi^+_{=n}} H^1( \tS, L_{\alpha}).
  \end{equation*}
  Since this $R$-module has finite length by Lemma~\ref{lem:h1},
  automorphisms of $\T$ that are congruent to the identity modulo
  $\m^j$ for sufficiently large $j$ can be lifted step by step to
  automorphisms of each $\B_{\le n}$ and of $\B$.
\end{proof}

\begin{proposition}\label{prop:B-torsor_unique_rigid}
  If $k$ has very good characteristic for $\Phi$, then all lifts
  $\B$ of $\B_{\le 1}$ are isomorphic as $B$-torsors, and satisfy
  $H^1(\tS, \ad(\B)) = 0$.
\end{proposition}

\begin{proof}
  We compare the $B$-torsor $\B$ chosen above to another lift $\B'$ of $\B_{\le 1}$.
  Let $\B_{\le n}'$ and $\T'$ denote the $B_{\le n}$-torsor and the $T$-torsor
  induced by $\B'$, respectively.  Here $\T$ and $\T'$ are isomorphic, but we
  will use various isomorphisms between them.

  Part (i) of Proposition~\ref{prop:B_unique} allows us to choose an isomorphism
  \begin{equation*}
    \phi_1: \B_k = \B_{|\tS_k} \to \B_{|\tS_k}' = \B_k'.
  \end{equation*}
  We claim that $\phi_1$ can be lifted to a compatible system of isomorphisms
  \begin{equation*}
    \phi_i: \B_{|i \tS_k} \to \B_{|i \tS_k}'
  \end{equation*}
  for $i \ge 1$. Indeed, the obstruction against lifting $\phi_{i-1}$ to $\phi_i$ is an element in
  \begin{equation*}
    H^1(\tS_k, \ad( \B_k) )
  \end{equation*}
  due to \cite[Th\'eor\`eme~VII.2.4.4]{MR0491681}, since the normal bundle of $\tS_k$ in $\tS$ is trivial.
  Therefore, part (ii) of Proposition~\ref{prop:B_unique} allows us to lift $\phi_{i-1}$ to $\phi_i$. Let
  \begin{equation*}
    \phi_{i, \le 0} \colon \T_{|i \tS_k} \to \T_{|i \tS_k}' \quad\text{and}\quad 
    \phi_{i, \le n} \colon \B_{\le n|i \tS_k} \to \B_{\le n|i \tS_k}'
  \end{equation*}
  denote the isomorphisms of torsors induced by $\phi_i$.

  The choice of an isomorphism $\T \to \T'$ induces bijections
  \begin{equation*}
    \Isom( \T, \T') \cong T( R) \quad\text{and}\quad
    \Isom( \T_{|i \tS_k}, \T_{|i \tS_k}') \cong T( R/\m^i)
  \end{equation*}
  for $i \ge 1$. Therefore, the restriction map
  \begin{equation*}
    \Isom( \T, \T') \to \Isom( \T_{|i \tS_k}, \T_{|i \tS_k}')
  \end{equation*}
  is surjective. We choose an integer $j$ which is sufficiently large in the sense of Lemma~\ref{lem:j>>0},
  and lift $\phi_{j, \le 0}$ to an isomorphism
  \begin{equation*}
    \psi_{\le 0} \colon \T \to \T'.
  \end{equation*}
  Its restrictions $\psi_{\le 0, i}$ to $i \tS_k \subset \tS$ satisfy by construction
  \begin{equation} \label{eq:phi0=psi0}
    \psi_{\le 0, i} = \phi_{i, \le 0}
  \end{equation}
  for $i \le j$. For $i > j$, these differ by an automorphism of $\T$,
  which can be lifted to an automorphism of $\B$ due to Lemma~\ref{lem:j>>0}.
  Modifying $\phi_{j+1}, \phi_{j+2}, \dots$ by these automorphisms of $\B$,
  we can achieve \eqref{eq:phi0=psi0} for all $i \ge 1$.

  We show by induction over $n$ that $\psi_{\le 0}$ can be lifted to an isomorphism
  \begin{equation*}
    \psi_{\le n}: \B_{\le n} \to \B_{\le n}'
  \end{equation*}
  whose restrictions $\psi_{\le n, i}$ to $i \tS_k \subset \tS$ satisfy
  \begin{equation} \label{eq:phi=psi}
    \psi_{\le n, i} = \phi_{i, \le n}.
  \end{equation}
  The obstruction against lifting $\psi_{\le n-1}$ to an isomorphism $\psi_{\le n}$ lies in
  \begin{equation*}
    \bigoplus_{\alpha \in \Phi^+_{=n}} H^1( \tS, L_{\alpha}).
  \end{equation*}
  The restriction of this class to $i \tS_k \subset \tS$ vanishes, because $\psi_{\le n-1, i} = \phi_{i, \le n-1}$
  admits the lift $\phi_{i, \le n}$. But the canonical map
  \begin{equation*}
    H^1( \tS, L_{\alpha}) \to \lim\limits_{\longleftarrow} H^1( i \tS_k, L_{\alpha|i \tS_k})
  \end{equation*}
  is bijective by the Theorem on Formal Functions \cite[Theorem III.11.1]{MR0463157} and Lemma~\ref{lem:h1}.
  Therefore, we can lift $\psi_{\le n-1}$ to an isomorphism $\psi_{\le n}$. Its restrictions $\psi_{\le n, i}$
  differ from the isomorphisms $\phi_{i, \le n}$ by an element of
  \begin{equation*}
    \lim\limits_{\longleftarrow} \bigoplus_{\alpha \in \Phi^+_{=n}} H^0( i \tS_k, L_{\alpha|i \tS_k}).
  \end{equation*}
  But any such element vanishes, because
  \begin{equation*}
    \lim\limits_{\longleftarrow} H^0( i \tS_k, L_{\alpha|i \tS_k}) \cong H^0( \tS, L_{\alpha}) = 0
  \end{equation*}
  according to the Theorem on Formal Functions \cite[Theorem III.11.1]{MR0463157} and Lemma~\ref{lem:h1} again.
  This proves that the chosen lift $\psi_{\le n}$ automatically satisfies \eqref{eq:phi=psi},
  which completes the induction. Taking $n$ sufficiently large, $\psi_{\le n}$ is the required isomorphism
  from $\B_{\le n} = \B$ to $\B_{\le n}' = \B'$.

  Infinitesimal rigidity follows from Proposition~\ref{prop:B_unique}
  (ii) using the Semicontinuity Theorem
  \cite[Theorem~III.12.8]{MR0463157} and Grauert's Theorem
  \cite[Corollary~III.12.9]{MR0463157}.
\end{proof}

We still assume that the $B$-torsor $\B$ over $\tS$ is a lift of $\B_{\le 1}$.
Extending the structure group of $\B$ to $G$, we obtain a $G$-torsor $\G=\B
\times^B G$ over $\tS$.

\begin{cor}\label{cor:G-torsor_unique_rigid}
  If $k$ has very good characteristic for $\Phi$, then 
  $H^1(\tS, \ad(\G))=0$.
\end{cor}

\begin{proof}
  The vector bundle $\ad(\B) = \B \times^B \frakb$ is associated with
  the $B$-torsor $\B$ and the $B$-module $\frakb$. Similarly, $\ad(\G)
  = \G \times^G \frakg = \B \times^B \frakg$ is associated with $\B$
  and the $B$-module $\frakg$. The $B$-module $\frakg/\frakb$ has a
  composition series with composition factors $\frakg_{-\alpha}$ for
  $\alpha \in \Phi^+$. Therefore, the associated vector bundle
  $\ad(\G)/\ad(\B)$ has a composition series with composition factors
  $\B \times^B \frakg_{-\alpha} \cong L_{-\alpha}$. Using induction
  over this composition series and
  \begin{equation*}
    H^0(\tS, L_{-\alpha})=H^1(\tS,L_{-\alpha})=0
  \end{equation*}
  for $\alpha \in \Phi^+$, we conclude that the natural map
  \begin{equation}\label{eq:isom_ad_B_ad_G}
    H^1(\tS,\ad(\B)) \to H^1(\tS,\ad(\G))
  \end{equation}
  is an isomorphism.
\end{proof}

\begin{proposition} \label{prop:trivial_on_D}
  The $G$-torsor $\G$ is trivial on every $(-2)$-curve $D \subset \tS_k$.
\end{proposition}

\begin{proof}
  Let $\alpha \in \Delta$ be the class of $D$. The restriction $L_{\alpha|D}$
  is a line bundle of degree $(\alpha, \alpha) = -2$ on $D \cong \PP^1_k$, which implies
  \begin{equation}\label{eq:onedim_h1}
    H^1(D, L_{\alpha|D}) \cong H^1(\PP^1_k, \OO_{\PP^1_k}(-2)) \cong k.
  \end{equation}
  Tensoring the short exact sequence
  \begin{equation*}
    0 \to \OO_{\tS_k}(-D) \to \OO_{\tS_k} \to \OO_D \to 0
  \end{equation*}
  with the line bundle $L_{\alpha,k} \cong \OO_{\tS_k}(D)$, we get a short exact sequence
  \begin{equation*}
    0 \to \OO_{\tS_k} \to L_{\alpha,k} \to L_{\alpha|D} \to 0
  \end{equation*}
  of coherent sheaves on $\tS_k$. Since $H^i(\tS_k, \OO_{\tS_k})$ vanishes for $i=1,2$
  by their birational invariance \cite[Proposition~V.3.4]{MR0463157}, the associated
  long exact cohomology sequence shows that the restriction homomorphism
  \begin{equation}\label{eq:isom_h1}
    H^1(\tS_k, L_{\alpha,k}) \to H^1(D, L_{\alpha|D})
  \end{equation}
  is bijective. For $\beta \in \Phi^+$ with $\beta \neq \alpha$,
  the degree of $L_{\beta|D}$ on $D \cong \PP^1_k$ is
  \begin{equation*}
    (\alpha,\beta) = -\langle \beta^\vee, \alpha\rangle =: n \in \{-1,0,1\},
  \end{equation*}
  because $\alpha \neq \beta$ are roots in the simply laced root system $\Phi$. This implies
  \begin{equation}\label{eq:trivial_h1}
    H^1(D, L_{\beta|D}) \cong H^1(\PP^1_k, \OO_{\PP^1_k}(n)) = 0.
  \end{equation}

  Let $G_{\alpha} \subset G$ be the split reductive subgroup with
  the same maximal torus $T$ and only the two roots $\pm \alpha$.
  Then $B_{\alpha} := B \cap G_{\alpha}$ sits in an exact sequence
  \begin{equation*}
    0 \to U_\alpha \to B_{\alpha} \to T \to 0.
  \end{equation*}
  Let the $B_{\alpha}$-torsor $\B_{\alpha}$ on $\tS$ be the lift of the $T$-torsor $\T$
  corresponding to the class $c_\alpha$ chosen in \eqref{eq:c_alpha}.
  Let $\G_{\alpha}$ be the $G_{\alpha}$-torsor over $\tS$ induced by $\B_{\alpha}$.

  The $B$-torsor induced by $\B_{\alpha}$ becomes isomorphic to $\B$ when both are restricted to $D$,
  because there the lifting over each $U_\beta$ with $\beta \in \Phi^+ \setminus \{\alpha\}$
  is unique by \eqref{eq:trivial_h1}. Hence it suffices to prove that $\G_{\alpha}$ is trivial on $D$.

  In the case (i) of blow-ups of $\PP^2$, \cite[II.2(6)]{MR579026}
  shows that $\alpha = e_1-e_2$ for two classes $e_i \in
  \Lambda$ satisfying $(e_i,e_i)=-1$ and $(e_1,e_2)=0$. 
  Since their intersection matrix $(\begin{smallmatrix}
    -1 & 0 \\ 0 & -1
  \end{smallmatrix})$ is invertible over $\ZZ$, we can extend
  $e_1,e_2$ to a basis $e_1,\dots,e_{10-d}$ of $\Lambda$ with
  $(e_1,e_i)=(e_2,e_i)=0$ for all $i \ge 3$.

  In the case (ii) of $\PP^1$-bundles over $\PP^1$, with $d=8$, we
  have $\alpha = e_1-e_2$, where $e_1$ and $e_2$ in $\Lambda =
  \Pic(\tS)$ restrict to the classes of fiber and constant section in
  $\tS_K = \PP^1_K \times_K \PP^1_K \to \PP^1_K$, respectively. Here,
  $(e_i,e_i)=0$ and $(e_1,e_2)=1$. We note that $e_1,e_2$ is a basis
  of $\Lambda$.

  In both cases (i) and (ii), we have $\alpha=e_1-e_2$ and
  \begin{equation*}
    \langle \alpha^\vee , e_i \rangle = (-\alpha,e_i) =
    \begin{cases}
      1 & \text{for }i=1,\\
      -1 & \text{for }i=2,\\
      0 & \text{for }i\ge 3.
    \end{cases}
  \end{equation*}
  These descriptions of $\alpha$ and $\alpha^\vee$ allow us to extend
  the decomposition $T \cong \Gm^{10-d}$ given by $e_1, \dots,
  e_{10-d}$ to a decomposition
  \begin{equation*}
    G_{\alpha} \cong \GL_2 \times \Gm^{8-d}.
  \end{equation*}
  This also induces a decomposition of $B_{\alpha}$.

  Let $L_{e_i}$ be a line bundle on $\tS$ of class $e_i \in \Pic(\tS)
  = \Lambda$. Under the above decompositions, the $B_{\alpha}$-torsor
  $\B_{\alpha}$ corresponds to the $10-d$ line bundles $L_{e_i}$ over
  $\tS$ and the vector bundle extension
  \begin{equation*}
    0 \to L_{e_1} \to E \to L_{e_2} \to 0
  \end{equation*}
  of class $c_{\alpha} \in \Ext^1( L_{e_2}, L_{e_1}) \cong H^1( \tS, L_{\alpha})$,
  and the $G_{\alpha}$-torsor $\G_{\alpha}$ corresponds to
  the vector bundle $E$ and the line bundles $L_{e_3}, \dots, L_{e_{10-d}}$ over $\tS$.

  The restriction of $L_{e_i}$ to $D \cong \PP^1_k$ is a line bundle of degree $(\alpha, e_i)$.
  For $i \ge 3$, we have $(\alpha, e_i) = 0$, and therefore $L_{e_i|D}$ is trivial.
  Since $(\alpha, e_1) = -1$ and $(\alpha, e_2) = 1$ in both cases (i) and (ii),
  the restriction of $E$ to $D \cong \PP^1_k$ is given as an extension
  \begin{equation}\label{eq:E_D}
    0 \to \OO_{\PP^1_k}(-1) \to E_{|D} \to \OO_{\PP^1_k}(1) \to 0,
  \end{equation}
  whose class in $H^1( \PP^1_k, \OO_{\PP^1_k}(-2)) \cong k$ corresponds to the restricted class
  \begin{equation*}
    c_{\alpha|D} \in H^1( D, L_{\alpha|D})
  \end{equation*}
  under the isomorphism in \eqref{eq:onedim_h1}. This class is nontrivial since the class
  \begin{equation*}
    c_{\alpha|\tS_k} = c_{\alpha,k} \in H^1( \tS_k, L_{\alpha,k})
  \end{equation*}
  is nontrivial by the choice of $c_{\alpha}$ in \eqref{eq:c_alpha} together with Lemma~\ref{lem:h1},
  and the restriction map from $\tS_k$ to $D$ in \eqref{eq:isom_h1} is bijective.
  Therefore, the extension \eqref{eq:E_D} does not split.
  This implies that the vector bundle $E_{|D}$ over $D \cong \PP^1_k$ is trivial.
  Hence the $G_{\alpha}$-torsor $\G_{\alpha|D}$ over $D$ is also trivial, as required.
\end{proof}

\begin{cor}
  Let $x$ be a singular point on $S_k$. The $G$-torsor $\G$ over $\tS$ constructed above
  becomes trivial over the following fiber product $\tS_x$:
  \begin{equation}\label{eq:S_x}
    \xymatrix{ 
      \tS_x \ar[r] \ar[d]_{\phi_x} & \tS \ar[d]^{\phi}\\
       \Spec(\OOhat_{S,x}) \ar[r] & S
    }
  \end{equation}
\end{cor}

\begin{proof}
  We work with the sequence of effective divisors on $\tS_k$
  \begin{equation*}
    0 = Z_0 < Z_1 < \dots < Z_r = \Zred < \dots < Z_N = Z
  \end{equation*}
  from Lemma~\ref{lem:combinatorial}, where $Z$ is still the fundamental cycle on $\tS_k$ over $x$.

  First, we show that $\G$ is trivial on $Z_j$ for $j=1, \dots, r$.
  Indeed, by induction, we can find a trivialization of $\G$ on $Z_{j-1}$. Then
  \begin{equation*}
    Z_j=Z_{j-1}\cup D_{i_j}
  \end{equation*}
  where $D_{i_j}$ meets $Z_{j-1}$ in at most one point.
  Proposition~\ref{prop:trivial_on_D} states that $\G$ is trivial on $D_{i_j}$.
  We can trivialize on $D_{i_j}$ in such a way that both trivializations agree on $Z_{j-1} \cap D_{i_j}$.
  Then they define a trivialization of $\G$ on $Z_j$.

  Next, we show by induction that $\G$ is trivial on $Z_j$ for all $j=r+1 , \dots, N$.
  Since $\G$ is trivial on $D_{i_j}$ by Proposition~\ref{prop:trivial_on_D}, its adjoint vector bundle
  \begin{equation*}
    \ad( \G) \to \tS
  \end{equation*}
  is also trivial on $D_{i_j}$. Therefore, Lemma~\ref{lem:Z_j} implies that
  \begin{equation}\label{eq:h1ad}
    H^1(D_{i_j}, \II_{Z_{j-1} \subset Z_j} \otimes \ad(\G)_{|D_{i_j}}) = 0.
  \end{equation}
  Assuming by induction that $\G$ is trivial on $Z_{j-1}$, the
  vanishing of \eqref{eq:h1ad} means that $\G$ is also trivial on
  $Z_j$ \cite[Th\'eor\`eme~VII.2.4.4]{MR0491681}. 

  In particular, $\G$ is trivial on $Z$. Therefore, Proposition~\ref{prop:I^n/I^n+1} implies that
  \begin{equation}\label{eq:h1ad_Z}
    H^1(Z, \II_{Z \subset \tS}^n/\II_{Z \subset \tS}^{n+1} \otimes \ad(\G)_{|Z}) = 0.
  \end{equation}
  Let $\m_x \subset \OO_S$ denote the ideal sheaf of $x$. We have
  \begin{equation*}
    \II_{Z \subset \tS} = \phi^*( \m_x)
  \end{equation*}
  according to \cite[Theorem~4]{MR0199191}, and therefore
  \begin{equation*}
    \II_{Z \subset \tS}^n = \phi^*( \m_x^n).
  \end{equation*}
  Let $Z^{(n)}$ denote the closed subscheme in $\tS$ with this ideal sheaf. 
  Assuming by induction that we have a section of $\G$ over $Z^{(n)}$, the
  vanishing of \eqref{eq:h1ad_Z} means that this section can be
  extended to a section of $\G$ over $Z^{(n+1)}$.

  These compatible sections induce a section of $\G$ over $\tS_x$
  by Grothendieck's Existence Theorem \cite[Scholie~5.4.2]{MR0217085},
  since $\tS$ is proper over $S$.
\end{proof}

Recall that we have lifted a universal torsor $\T$ over $\tS$ nontrivially to
a $B_{\le 1}$-torsor $\B_{\le 1}$, and further to a $B$-torsor $\B$; see
Lemma~\ref{lem:B1_unique} and Lemma~\ref{lem:lift_to_B}.

\begin{theorem}\label{thm:main}
  Let $\G$ still be the $G$-torsor over $\tS$ induced by $\B$.
  \begin{enumerate}
  \item[(i)] There is a unique $G$-torsor $\G'$ over $S$
    such that $\phi^*\G' \cong \G$.
  \item[(ii)] If $k$ has very good characteristic for $\Phi$,
    then $H^1(S, \ad(\G'))=0$.
  \end{enumerate}
\end{theorem}

\begin{proof}
  Since $\G$ is an affine scheme over $\tS$, we have
  \begin{equation*}
    \G \cong \relSpec_{\OO_{\tS}}(\A)
  \end{equation*}
  for some quasicoherent $\OO_{\tS}$-algebra $\A$. We define
  \begin{equation*}
    \G' := \relSpec_{\OO_S}(\phi_* \A).
  \end{equation*}
  The adjunction morphism $\phi^*\phi_*\A \to \A$ induces a natural map
  \begin{equation}\label{eq:adjunction}
    \G \to \G' \times_S \tS.
  \end{equation}
  Assume that $G$ is the spectrum of the $R$-algebra $A$. The group action
  \begin{equation*}
    G \times_R \G \to \G
  \end{equation*}
  induces a morphism $\A \to A \otimes_R \A$ of $\OO_\tS$-algebras, and hence a morphism
  \begin{equation*}
    \phi_*\A \to \phi_*(A \otimes_R \A)=A \otimes_R (\phi_* \A)
  \end{equation*}
  of $\OO_S$-algebras. Here, the last equality holds because $G$, and hence also $A$, is flat over $R$.
  This morphism of $\OO_S$-algebras induces a morphism
  \begin{equation} \label{eq:action_on_H}
    G \times_R \G' \to \G'
  \end{equation}
  over $S$. We claim that the following statements hold, which imply (i):
  \begin{itemize}
   \item The morphism \eqref{eq:action_on_H} is a group action of $G$ on $\G'$ over $S$.
   \item This group action turns $\G'$ into a $G$-torsor over $S$.
   \item The natural map \eqref{eq:adjunction} is an isomorphism of $G$-torsors.
  \end{itemize}
  According to \cite[Propositions~2.5.1 and 2.7.1]{MR0199181}, all this can be tested
  locally in the fpqc-topology on $S$. We use the fpqc-covering
  \begin{equation*}
    (S \setminus \Ssing_k) \amalg \coprod_{x \in \Ssing_k} \Spec( \widehat\OO_{S,x}) \to S
  \end{equation*}
  where $\Ssing_k \subset S_k \subset S$ denotes the singular locus of $S_k$.

  All our claims hold over $S \setminus \Ssing_k$ because $\phi$ is an isomorphism there.
  They also hold over each $\Spec( \widehat\OO_{S,x})$ because $\G$ is trivial there, and
  \begin{equation*}
    (\phi_x)_*\OO_{\tS_x} = \OO_{\Spec(\OOhat_{S,x})}
  \end{equation*}
  by Proposition~\ref{prop:O_tS_O_S} and flat base change in the diagram
  \eqref{eq:S_x}.

  Uniqueness of $\G'$ also follows from Proposition~\ref{prop:O_tS_O_S}.

  For the proof of (ii), we note that $\ad(\G) = \phi^*\ad(\G')$
  by construction. Therefore, we have $R^i\phi_*(\ad\G) = 0$ for
  all $i > 0$ since this can be tested Zariski locally on $S$,
  where it holds by Proposition~\ref{prop:O_tS_O_S}. Using the
  Leray spectral sequence, we conclude that
  \begin{equation}\label{eq:isom_ad_G_ad_G'}
    H^1(S, \ad(\G')) \cong H^1(\tS, \ad(\G)).
  \end{equation}
  Hence (ii) follows from Corollary~\ref{cor:G-torsor_unique_rigid}.
\end{proof}

\begin{remark}
  The restriction of $\G$ to the generic fiber $S_K$ is induced by the
  $T$-torsor $\T$.  But over the special fiber $S_k$, the restriction of $\G$
  does not come from a $T$-torsor.  The universal $T$-torsor over the
  desingularization $\tS_k$ is nontrivial on the $(-2)$-curves, and therefore
  does not descend to $S_k$.
\end{remark}

\begin{proposition}\label{prop:uniqueness}
  Let the $B$-torsor $\Bbar$ over $\tS$ be an arbitrary lift of
  $\T$. Let $\Gbar$ be the $G$-torsor over $\tS$ induced by
  $\Bbar$. Suppose that $\Gbar$ descends to a $G$-torsor
  $\Gbar'$ over $S$.  If $k$ has very good characteristic for $\Phi$, then
  there is an automorphism $\sigma$ of $G$ with $\sigma_{|T} = \id_T$ such that
  \begin{equation*}
    \Bbar \cong (\sigma_{|B})_* \B, \qquad
    \Gbar \cong \sigma_* \G \qquad\text{and}\qquad
    \Gbar' \cong \sigma_* \G'.
  \end{equation*}
\end{proposition}

\begin{proof}
  Let $\Bbar_{\le 1} = \Bbar \times^B B_{\le 1}$. Let
  $\cbar_\alpha \in H^1(\tS, L_\alpha)$ for $\alpha \in \Delta$ be the
  extension classes corresponding to $\Bbar_{\le 1}$ as in
  (\ref{eq:c_alpha}).

  Suppose that each $\cbar_\alpha$ generates $H^1(\tS,L_\alpha)$.
  Then $\Bbar_{\le 1} \cong (\sigma_{|B_{\le 1}})_* \B_{\le 1}$ for
  some such automorphism $\sigma$ of $G$ by Lemma~\ref{lem:B1_unique}, and hence
  $\Bbar \cong (\sigma_{|B})_* \B$ by Proposition~\ref{prop:B-torsor_unique_rigid}. Therefore,
  $\Gbar \cong \sigma_* \G$. By the uniqueness in Theorem~\ref{thm:main} (i), this implies
  $\Gbar' \cong \sigma_* \G'$.

  Now suppose that $\cbar_\alpha$ does not generate $H^1(\tS,L_\alpha)$ for
  one $\alpha \in \Delta$. Let $D$ be the corresponding $(-2)$-curve on
  $\tS_k$. The $B$-torsor $\T \times^T B$ becomes isomorphic to $\Bbar$ when
  both are restricted to $D$, because there $\cbar_\alpha$ vanishes and the
  lifting over each $U_\beta$ with $\beta \in \Phi^+\setminus \{\alpha\}$ is
  unique by (\ref{eq:trivial_h1}). Hence $\Gbar_{|D} \cong \T_{|D} \times^T G$,
  and therefore  
  \begin{equation*}
    \ad(\Gbar)_{|D} \cong \T_{|D} \times^T \frakg \cong \OO_{\PP^1_k}^{10-d}
    \oplus \bigoplus_{\beta \in \Phi} \OO_{\PP^1_k}((\beta,\alpha)).
  \end{equation*}
  The integer $(\beta,\alpha)$ is nonzero at least for $\beta=\alpha$, so the
  vector bundle $\ad(\Gbar)_{|D}$ is nontrivial by the Krull--Remak--Schmidt
  theorem. Hence $\Gbar_{|D}$ is nontrivial, contradicting the assumption that
  $\Gbar$ descends to $S$.
\end{proof}

The main theorem in the introduction follows from these results: The descent
statement is contained in Theorem~\ref{thm:main} (i), and the uniqueness in
Proposition~\ref{prop:uniqueness}. The claims about infinitesimal rigidity
follow from Proposition~\ref{prop:B-torsor_unique_rigid},
Corollary~\ref{cor:G-torsor_unique_rigid} and Theorem~\ref{thm:main} (ii).

\section{Infinitesimal rigidity in one bad characteristic}
\label{sec:d4}

In the setting of Section~\ref{sec:del_pezzo}, we assume that the residue
field $k$ of $R$ is of characteristic $2$, and that $S$ is a family of
cubic surfaces over $R$ whose special fiber $S_k$ has one singularity,
which is of type $\mathbf{D}_4$.

For the geometry of cubic surfaces with a $\mathbf{D}_4$-singularity,
which were already studied by Schl\"afli \cite{10.2307/108795},
see \cite[\S 4]{MR2029868}, for example.
Up to the action of the Weyl groups, a root system of type $\mathbf{D}_4$
admits only one embedding into one of type $\mathbf{E}_6$.
Choosing one particular embedding allows us to describe $S_k$ as follows.
We may assume that its minimal desingularization
\begin{equation*}
  \tS_k \to S_k
\end{equation*}
is obtained from $\PP^2_k$ by blowing up three points
$x_1,x_2,x_3$ on a line and then three points $x_4,x_5,x_6$, where
$x_{i+3}$ lies on the $i$-th exceptional divisor, for $i=1,2,3$. Let
$h \in \Lambda$ be the pullback of $[\OO_{\PP^2_k}(1)]$, and let
$e_i \in \Lambda$ be the class of (the total
transform of) the $i$-th exceptional divisor $E_i$, for $i=1, \dots,
6$. Then the classes of the $(-2)$-curves are
\begin{equation*}
  \alpha_i=e_i-e_{i+3}\quad(i=1,2,3),\qquad \alpha_4=h-e_1-e_2-e_3
\end{equation*}
with the following Dynkin diagram:
\begin{equation*}
  \xymatrix@R=0.2in @C=0.2in{\alpha_1 \ar@{-}[r] & \alpha_4 \ar@{-}[r]\ar@{-}[d]
    & \alpha_2\\ & \alpha_3}
\end{equation*}
In particular, $\Phi$ is indeed a root system of type $\mathbf{D}_4$. 
The surface $S_k$ contains six lines. Three of them, namely the images of $E_4$, $E_5$, and $E_6$,
meet in the singularity. The other three lines $\ell_i$ for $i = 1, 2, 3$
are the images of curves in $\tS_k$ of class $h - e_i - e_{i+3}$; they may or may not meet in one
point on $S_k$.

Let $\B$ be a lift as in Lemma~\ref{lem:lift_to_B} of a universal
torsor $\T$, and let $\B_k:=\B_{|\tS_k}$.
In this situation, $\B_k$ may or may not be infinitesimally rigid:

\begin{proposition}\label{prop:D4_ad}
  We have
  \begin{equation*}
    h^1(\tS_k, \ad(\B_k))=
    \begin{cases}
      0 &\text{if $\ell_1\cap\ell_2\cap\ell_3 = \emptyset$ on $S_k$},\\
      1 &\text{otherwise}.
    \end{cases}
  \end{equation*}
\end{proposition}

To prove this, we follow the strategy of Lemma~\ref{lem:cup_epi} and
Proposition~\ref{prop:B_unique} (ii). These and what follows take place only
in the special fiber $\tS_k$. Replacing $k$ by its algebraic closure
and $\B_k$ by its base change, we may assume that $k$ is algebraically
closed. By Bertini's theorem \cite[Remark~II.8.18.1]{MR0463157}, we
can intersect an anticanonical embedding $S_k \subset \PP^3_k$ with a
suitable plane in $\PP^3_k$ to obtain a smooth curve of degree $3$ in
that plane, not containing the $\mathbf{D}_4$ singularity. Its
preimage $C \subset \tS_k$ is an elliptic curve.

Since $k$ is algebraically closed, $S_k$ is isomorphic to the surface defined by
\begin{equation}\label{eq:D4_cubic_ordinary}
  x_0(x_1+x_2+x_3)^2-x_1x_2x_3=0
\end{equation}
if $\ell_1\cap\ell_2\cap\ell_3=\emptyset$, and to the surface defined by
\begin{equation}\label{eq:D4_cubic_supersingular}
  x_0(x_1+x_2+x_3)^2+x_1x_2(x_1+x_2)=0
\end{equation}
if $\ell_1\cap\ell_2\cap\ell_3\ne\emptyset$; see
\cite[Lemma~4]{MR80f:14021} and \cite[Remark~4.1]{MR2029868}. In both
cases, the singularity is $(1:0:0:0)$, hence a plane not containing it
is defined by $x_0=a_1x_1+a_2x_2+a_3x_3$ for some $a_1,a_2,a_3 \in
k$. Whenever its intersection with $S_k$ is an elliptic curve $C$,
\cite[Proposition~IV.4.21]{MR0463157} shows that $C$ is ordinary in
case (\ref{eq:D4_cubic_ordinary}) and supersingular in case
(\ref{eq:D4_cubic_supersingular}).

\begin{lemma}\label{lem:Hi_elliptic}
  Let $\alpha \in \Phi^+$ be a positive root. Then the restriction
  maps $H^i(\tS_k, L_{\alpha,k}) \to H^i(C,L_{\alpha,k|C})$ are
  isomorphisms.
\end{lemma}

\begin{proof}
  As in \cite[Lemma~3.3]{MR1941576}, the long exact sequence arising
  from
  \begin{equation*}
    0 \to \OO_{\tS_k}(-C) \otimes L_{\alpha,k} \to L_{\alpha,k} \to L_{\alpha,k|C} \to 0
  \end{equation*}
  together with the Serre duality isomorphisms
  \begin{equation*}
    H^i(\tS_k,\OO_{\tS_k}(-C)\otimes L_{\alpha,k}) \to H^{2-i}(\tS_k,
    L_{-\alpha,k})^\vee=0
  \end{equation*}
  give the result.
\end{proof}

Let $U_{\le n}$ be the kernel of the canonical projection $B_{\le n}
\to T$. The resulting short exact sequence
\begin{equation*}
  0 \to U_{\le n} \to B_{\le n} \to T \to 1
\end{equation*}
induces the short exact sequence of Lie algebras
\begin{equation}\label{eq:seq_lie}
  0 \to \fraku_{\le n} \to \frakb_{\le n} \to \frakt \to 0.
\end{equation}
Let $\ad_n(\B_{\le n,k})$ be the vector bundle associated with the
$B_{\le n,k}$-torsor $\B_{\le n,k}$ via the $B_{\le n,k}$-module
$\fraku_{\le n,k}$. Using~(\ref{eq:seq_lie}), we obtain the short
exact sequence
\begin{equation}\label{eq:ad_n_to_ad}
  0 \to \ad_n(\B_{\le n,k}) \to \ad(\B_{\le n,k}) \to \Lambda^\vee \otimes_\ZZ \OO_{\tS_k} \to 0.
\end{equation}

\begin{lemma}\label{lem:S_to_C}
  The restriction maps
  \begin{equation*}
    H^i(\tS_k, \ad_n(\B_{\le n,k})) \to H^i(C,\ad_n(\B_{\le n,k})_{|C})
  \end{equation*}
  are isomorphisms.
\end{lemma}

\begin{proof}
  Since $\ad_n(\B_{\le n,k})$ has a composition series with
  composition factors $L_\alpha$ for some $\alpha \in \Phi^+$, this
  follows from Lemma~\ref{lem:Hi_elliptic} and the five lemma.
\end{proof}

For every $\alpha \in \Phi$, we denote by $\exp_\alpha$ the
exponential map from the underlying additive group of $\frakg_\alpha$
onto $U_\alpha \subset G$.

\begin{lemma}\label{lem:exp_A2}
  Let $\alpha,\beta \in \Phi$ be nonproportional roots.  Then the
  adjoint action of $U_\alpha$ on $\frakg$ satisfies
  \begin{equation*}
    \exp_\alpha(x)\cdot y = y+[x,y]
  \end{equation*}
  for all $x \in \frakg_\alpha$ and $y \in \frakg_\beta$, where $[x,y]
  \in \frakg_{\alpha+\beta}$ is $0$ if $\alpha+\beta \notin \Phi$.
\end{lemma}

\begin{proof}
  Let $G' \subset G$ be the centralizer of the reduced identity
  component of $\ker(\alpha) \cap \ker(\beta) \subset T$. Then $G'$ is
  a reductive group with maximal torus $T$ and root datum $\Phi \cap
  \langle \alpha,\beta \rangle \subset \Lambda$ of semisimple rank $2$.
  We have $U_\alpha \subset G'$.

  Let $G'' = G'/Z'$, where $Z'$ is the center of $G'$.  Since the
  exponential maps $\exp_\alpha$ for $G, G', G''$ agree, it suffices
  to prove the claim in $G''$. Since $G''$ is isomorphic to $\PGL_3$
  or $(\PGL_2)^2$, this is an easy computation.
\end{proof}

By \cite[Theorem~5(i)]{MR0131423}, there is an indecomposable vector
bundle on $C$ of rank $r$ and degree $0$, unique up to
isomorphism, which we denote by $\F_r$.

\begin{proposition}\label{prop:D4_ad_C}
  We have $\ad_n(\B_{\le 3,k})_{|C} \cong \F_3^2\oplus \F_2\oplus
  \Frob^*\F_2$, where $\Frob: C \to C$ is the (absolute) Frobenius morphism.
\end{proposition}

\begin{proof}
  Let $B^{\ad}$ be the quotient of $B$ modulo the center $Z$ of
  $G$. Note that $\Hom(B^{\ad}, \Gm)=\langle \Phi\rangle \subset
  \Lambda$. Let $B_{\le n}^{\ad}$ denote the quotient of $B_{\le n}$
  modulo the image of $Z$. Let $\B_{\le n}^{\ad}$ denote the torsor
  induced from $\B_{\le n}$ by extension of structure group along the
  projection $B_{\le n} \to B_{\le n}^{\ad}$.

  The action of $B_{\le 3}$ on $\fraku_{\le 3}$ factors through the
  quotient $B_{\le 2}^{\ad}$ of $B_{\le 3}$. 
  Let $B_{\PGL_3}$ be the standard Borel subgroup of classes of upper
  triangular matrices in $\PGL_3$ over $R$. The three embeddings of
  the Dynkin diagram $\mathbf{A}_2$ into $\mathbf{D}_4$ yield three
  group homomorphisms $p_i: B \to B_{\PGL_3}$ as follows.

  The Lie algebra of $B$ has the root space decomposition
  \begin{equation*}
    \frakb = \frakt \oplus \bigoplus_{\alpha \in \Phi^+} \frakg_\alpha.
  \end{equation*}
  For $i=1, \dots, 4$, we choose nonzero $x_{\alpha_i} \in
  \frakg_{\alpha_i}$. For $i=1,2,3$, we define
  \begin{equation*}
    x_{\alpha_i+\alpha_4}
    := [x_{\alpha_4},x_{\alpha_i}] \in \frakg_{\alpha_i+\alpha_4}.
  \end{equation*}
  For $i \ne j \in \{1, 2, 3\}$, we define
  \begin{equation*}
    x_{\alpha_i+\alpha_j+\alpha_4} :=
    [x_{\alpha_j+\alpha_4},x_{\alpha_i}]=[x_{\alpha_i+\alpha_4},x_{\alpha_j}] \in \frakg_{\alpha_i+\alpha_j+\alpha_4}.
  \end{equation*}
  As in (\ref{eq:epsilon_relation}), we have followed the sign convention
  from \cite[12.14]{MR1642713} here, i.e.,
  \begin{equation*}
     [x_\alpha,x_\beta] = \epsilon_{\alpha,\beta} x_{\alpha+\beta}
  \end{equation*}
  whenever $\alpha,\beta,\alpha+\beta \in \Phi^+$, where in this case
  $\epsilon_{\alpha,\beta} = (-1)^{f(\alpha,\beta)}$ for the bilinear form $f$
  defined by
  \begin{equation*}
    f(\alpha_i,\alpha_j) =
    \begin{cases}
      (\alpha_i,\alpha_j), &i<j,\\
      \frac 1 2 (\alpha_i,\alpha_i), &i=j,\\
      0, &i>j
    \end{cases}
  \end{equation*}
  for $i,j \in \{1,\dots, 4\}$. Note that this turns out to be
  independent of the ordering of $\alpha_1,\alpha_2,\alpha_3$.

  For $i=1,2,3$, let $p_i: B \to B_{\PGL_3}$ be the surjective
  homomorphism that vanishes on $U_\alpha$ for all $\alpha \in
  \Phi^+\setminus \{\alpha_i,\alpha_4,\alpha_i+\alpha_4\}$ and on
  $\ker(\alpha_i) \cap \ker(\alpha_4) \subset T$. Note that
  $\ker(\alpha_i) \cap \ker(\alpha_4)$ is central modulo these
  $U_\alpha$, so the subgroup generated by these $U_\alpha$ and
  $\ker(\alpha_i) \cap \ker(\alpha_4)$ is normal in $B$.  More
  precisely, $p_i$ corresponds to the Lie algebra homomorphism $\frakb
  \to \frakb_{\PGL_3}$ defined by
  \begin{align*}
    \frakt = \Lambda^\vee \otimes R &\to \frakb_{\PGL_3},\\
    \rho \otimes 1 &\mapsto
    \begin{pmatrix}
      \langle \rho,\alpha_i+\alpha_4 \rangle & 0 & 0\\
      0 & \langle \rho,\alpha_i \rangle & 0\\
      0 & 0 & 0
    \end{pmatrix}
  \end{align*}
  and
  \begin{align*}
    \frakg_{\alpha_i} &\to \frakb_{\PGL_3},&\frakg_{\alpha_4} &\to \frakb_{\PGL_3},&\frakg_{\alpha_i+\alpha_4} &\to \frakb_{\PGL_3},\\
    x_{\alpha_i} &\mapsto
    \begin{pmatrix}
      0 & 0 & 0\\
      0 & 0 & 1\\
      0 & 0 & 0
    \end{pmatrix}&
    x_{\alpha_4} &\mapsto
    \begin{pmatrix}
      0 & 1 & 0\\
      0 & 0 & 0\\
      0 & 0 & 0
    \end{pmatrix},&
    x_{\alpha_i+\alpha_4} &\mapsto
    \begin{pmatrix}
      0 & 0 & 1\\
      0 & 0 & 0\\
      0 & 0 & 0
    \end{pmatrix},   
  \end{align*}
  vanishing on all other $\frakg_\alpha$ with $\alpha \in \Phi^+$.
  
  By construction, the product $(p_1,p_2,p_3) : B \to B_{\PGL_3}^3$
  factors through a homomorphism
  \begin{equation}\label{eq:hom_BPGL3}
    B_{\le 2}^{\ad} \to B_{\PGL_3}^3.
  \end{equation}
  Let $p : B_{\PGL_3} \to B_{\PGL_2}$ be the projection $(a_{r,s})_{1
    \le r,s \le 3} \mapsto (a_{r,s})_{1 \le r,s \le 2}$ onto the upper
  left $(2\times 2)$-minor. Then (\ref{eq:hom_BPGL3}) induces an isomorphism
  \begin{equation}\label{eq:B_2_ad}
    B_{\le 2}^{\ad} \cong \{(b_1,b_2,b_3) \in B_{\PGL_3}^3 \mid p(b_1)=p(b_2)=p(b_3)\}.
  \end{equation}
  Note that $U_{\le 2}$ can be identified with the unipotent radical
  of $B_{\le 2}^{\ad}$.  Identifying also the unipotent radical
  $U_{\GL_3}$ of the standard Borel subgroup $B_{\GL_3} \subset \GL_3$
  with the unipotent radical of $B_{\PGL_3}$, we obtain an isomorphism
  \begin{equation*}
    U_{\le 2} \cong \{(u_1,u_2,u_3) \in U_{\GL_3}^3 \mid p(u_1)=p(u_2)=p(u_3)\}.
  \end{equation*}

  We choose $e_\alpha$ and $f_\alpha$ for $\alpha \in \Phi^+$ as in
  Lemma~\ref{lem:e_alpha-f_alpha}. The sections $e_\alpha$ define isomorphisms
  $L_{\alpha|C} \cong \OO_C$ since $C$ does not intersect the vanishing locus
  of $e_\alpha$, which consists of $(-2)$-curves. This gives a reduction of
  structure group of $\B_{\le 2|C}^{\ad}$ to a $U_{\le 2|C}$-torsor
  $\U_{\le 2}$, which we have only on $C$.

  Let $\V$ be the vector bundle of rank $2$ over $C$ associated with
  the $U_{\le 2|C}$-torsor $\U_{\le 2}$ via the common composition
  \begin{equation*}
    p\circ p_1=p\circ p_2=p\circ p_3: U_{\le 2} \to U_{\GL_3} \to U_{\GL_2} \subset \GL_2.
  \end{equation*}
  Using the above identifications, $\V$ is by construction an extension
  \begin{equation*}
    0 \to \OO_C \to \V \to \OO_C \to 0
  \end{equation*}
  whose class $c_4 \in H^1(C, \OO_C)$ is given by
  \begin{equation*}
    e_{\alpha_4|C} \cdot c_4 = c_{\alpha_4|C} = f_{\alpha_4|C}.
  \end{equation*}
  Since $c_{\alpha_4|C}$ is nontrivial, we can identify the extension $\V$
  with the Atiyah bundle $\F_2$. The composition $\F_2 \cong \V \to \OO_C$
  induces an isomorphism $H^1(C, \F_2) \to H^1(C, \OO_C)$.

  For $i=1,2,3$, let $\V_i$ be the vector bundle of rank $3$ over $C$
  associated with $\U_{\le 2}$ via
  \begin{equation*}
    p_i : U_{\le 2} \to U_{\GL_3} \subset \GL_3.
  \end{equation*}
  Using the above identifications, $\V_i$ is by construction an extension
  \begin{equation*}
    0 \to \F_2 \to \V_i \to \OO_C \to 0
  \end{equation*}
  whose class $c_{\alpha_i|C} \in H^1(C, \OO_C) = H^1(C, \F_2) =
  \Ext^1(\OO_C, \F_2)$ is given by
  \begin{equation*}
    e_{\alpha_i|C} \cdot c_i = c_{\alpha_i|C} = f_{\alpha_i|C}.
  \end{equation*}
  Applying $[e_{\alpha_j+\alpha_4},\_]$ to this equation, and using that
  \begin{align*}
    [e_{\alpha_j+\alpha_4},e_{\alpha_i}] &= e_{\alpha_i+\alpha_j+\alpha_4} =
    [e_{\alpha_i+\alpha_4},e_{\alpha_j}],\\
    [e_{\alpha_j+\alpha_4},f_{\alpha_i}] &= f_{\alpha_i+\alpha_j+\alpha_4} =
    [e_{\alpha_i+\alpha_4},f_{\alpha_j}]
  \end{align*}
  by Lemma~\ref{lem:e_alpha-f_alpha}, we conclude that
  \begin{equation}\label{eq:c_alpha_C}
    c_1=c_2=c_3 \in H^1(C,\OO_C).
  \end{equation}
  This allows us to identify
  $\V_1$, $\V_2$, $\V_3$ as extensions of $\OO_C$ by $\F_2$. This
  identification reduces $\U_{\le 2}$ to a torsor $\U_{\GL_3}$ under
  the diagonally embedded subgroup
  \begin{equation}\label{eq:U_diagonal}
    U_{\GL_3} \subset U_{\le 2}.
  \end{equation}
  Since the class in (\ref{eq:c_alpha_C}) is nontrivial, we note that
  $\V_i \cong \F_3$. 

  Similarly, $c_3=c_4 \in H^1(C,\OO_C)$ because
  \begin{align*}
    [e_{\alpha_4},e_{\alpha_3}] &= e_{\alpha_3+\alpha_4} = -[e_{\alpha_3},e_{\alpha_4}],\\
    [e_{\alpha_4},f_{\alpha_3}] &= f_{\alpha_3+\alpha_4} = -[e_{\alpha_3},f_{\alpha_4}].
  \end{align*}
  This allows us to identify the subbundle $\F_2 \cong \V \subset \V_3$ with
  the quotient $\V_3/\OO_C \cong \F_2$. This identification reduces
  $\U_{\GL_3}$ to a torsor $\U'$ under the subgroup
  \begin{equation*}
    U':=\left\{\left.
      \begin{pmatrix}
        1 & \lambda & \mu\\
        0 & 1 & \lambda\\
        0 & 0 & 1
      \end{pmatrix}
      \,\right|\, \lambda,\mu \in k\right\} \subset U_{\GL_3}.
  \end{equation*}

  The next step is to study $\fraku_{\le 3}$ as a ten-dimensional
  representation of these subgroups $U' \subset U_{\GL_3} \subset
  U_{\le 2}$. We have
  \begin{equation*}
    \fraku_{\le 3} = \bigoplus_{\alpha \in \Phi^+_{\le 3}} \frakg_\alpha,
  \end{equation*}
  with basis $(x_\alpha)$ as introduced above.

  For $\lambda \in k$, consider
  \begin{equation*}
    u_\lambda =
    \begin{pmatrix}
      1 & 0 & 0\\
      0 & 1 & \lambda\\
      0 & 0& 1
    \end{pmatrix}
    \in U_{\GL_3},\qquad
    v_\lambda =
    \begin{pmatrix}
      1 & \lambda & 0\\
      0 & 1 & 0\\
      0 & 0 & 1
    \end{pmatrix}
    \in U_{\GL_3}.
  \end{equation*}
  Its images under (\ref{eq:U_diagonal}) are
  \begin{equation*}
    \exp_{\alpha_1}(\lambda x_{\alpha_1})\exp_{\alpha_2}(\lambda
    x_{\alpha_2})\exp_{\alpha_3}(\lambda x_{\alpha_3}) \in U_{\le 2},\qquad \exp_{\alpha_4}(\lambda x_{\alpha_4}) \in U_{\le 2};
  \end{equation*}
  for the image of $u_\lambda$, the ordering does not matter since
  $[x_{\alpha_i},x_{\alpha_j}]=0$ for all $i,j \le 3$. For every
  $\alpha \in \Phi^+_{\le 3}$ and $y \in \frakg_\alpha$,
  Lemma~\ref{lem:exp_A2} then gives
  \begin{align*}
    u_\lambda \cdot y &= y + \sum_{i=1}^3 [\lambda x_{\alpha_i},y]+ 
    \sum_{1 \le i < j \le 3}[\lambda x_{\alpha_i},[\lambda x_{\alpha_j}, y]],\\
    v_\lambda \cdot y &= y + [\lambda x_{\alpha_4},y].
  \end{align*}
  Note that the last sum in the expression for $u_\lambda \cdot y$
  vanishes unless $\alpha = \alpha_4$.
  
  In particular, using the notation
  \begin{equation*}
    z_1:=\sum_{i=1}^3 x_{\alpha_i},\quad z_2:=\sum_{i=1}^3 x_{\alpha_i+\alpha_4}, 
    \quad z_3:=\sum_{1 \le i < j \le 3} x_{\alpha_i+\alpha_j+\alpha_4},
  \end{equation*}
  we have
\begin{align*}
    u_\lambda \cdot x_{\alpha_4} &= x_{\alpha_4} - \lambda z_2 + \lambda^2 z_3,\\
    u_\lambda \cdot x_{\alpha_j+\alpha_4} & = x_{\alpha_j+\alpha_4} - \lambda \sum_{i \in \{1,2,3\} \setminus \{j\}} x_{\alpha_i+\alpha_j+\alpha_4}\qquad  (j=1,2,3),\\
    v_\lambda \cdot x_{\alpha_j} &= x_{\alpha_j} + \lambda x_{\alpha_j+\alpha_4} \qquad (j=1,2,3),
  \end{align*}
  while $u_\lambda$ and $v_\lambda$ acts as the identity on all other
  $x_{\alpha}$. These imply that
  \begin{equation*}
    v_\lambda \cdot z_1 = z_1+\lambda z_2, \text{ and } 
    u_\lambda \cdot z_2 = z_2-2\lambda z_3 = z_2\text{ in characteristic $2$,}
  \end{equation*}
  while $u_\lambda$ and $v_\lambda$ act as the identity on all other
  $z_i$.  We observe that $\fraku_{\le 3}$ decomposes as
  $U_{\GL_3}$-module into the direct sum of the three vector spaces
  \begin{align*}
    \fraku_{\le 3}^1&:=\langle x_{\alpha_1}, x_{\alpha_1+\alpha_4}, x_{\alpha_1+\alpha_2+\alpha_4}+x_{\alpha_1+\alpha_3+\alpha_4}\rangle,\\
    \fraku_{\le 3}^2&:=\langle x_{\alpha_2}, x_{\alpha_2+\alpha_4}, x_{\alpha_1+\alpha_2+\alpha_4}+x_{\alpha_2+\alpha_3+\alpha_4}\rangle,\\
    \fraku_{\le 3}'&:=\langle x_{\alpha_4}, z_1, z_2, z_3\rangle.
  \end{align*}
  The vector bundle $\V^1$ of rank $3$ over $C$ associated with the
  $U_{\GL_3|C}$-torsor $\U_{\GL_3}$ via the representation
  $\fraku_{\le 3}^1$ is isomorphic to $\F_3$. Indeed, $\fraku_{\le
    3}^1$ has the composition series
  \begin{equation*}
    0 \subset \langle x_{\alpha_1+\alpha_2+\alpha_4}+x_{\alpha_1+\alpha_3+\alpha_4} \rangle \subset \langle x_{\alpha_1+\alpha_4}, x_{\alpha_1+\alpha_2+\alpha_4}+x_{\alpha_1+\alpha_3+\alpha_4}\rangle \subset \fraku_{\le
      3}^1.
  \end{equation*}
  The composition factors are the trivial one-dimensional
  representations.  Therefore, $\V^1$ is a double extension of
  $\OO_C$ by $\OO_C$ by $\OO_C$.  Here, both extensions of $\OO_C$ by
  $\OO_C$ are nontrivial because the corresponding representations
  $\langle x_{\alpha_1+\alpha_4},
  x_{\alpha_1+\alpha_2+\alpha_4}+x_{\alpha_1+\alpha_3+\alpha_4}\rangle$
  and $\fraku_{\le 3}^1/\langle
  x_{\alpha_1+\alpha_2+\alpha_4}+x_{\alpha_1+\alpha_3+\alpha_4}
  \rangle$ are the two two-dimensional standard representations of
  $U_{\GL_3}$. A similar argument shows that the vector bundle $\V^2$
  of rank $3$ over $C$ associated with the $U_{\GL_3|C}$-torsor
  $\U_{\GL_3}$ via the representation $\fraku_{\le 3}^2$ is isomorphic
  to $\F_3$.

  The subgroup $U' \subset U_{\GL_3}$ is generated by
  \begin{equation*}
    u_\lambda v_\lambda =
    \begin{pmatrix}
      1 & \lambda & 0\\
      0 & 1 & \lambda\\
      0 & 0& 1
    \end{pmatrix}
    \in U',\qquad
    [v_1,u_\mu] =
    \begin{pmatrix}
      1 & 0 & \mu\\
      0 & 1 & 0\\
      0 & 0 & 1
    \end{pmatrix}
    \in U'
  \end{equation*}
  for $\lambda,\mu \in k$. We have
  \begin{align*}
    u_\lambda v_\lambda\cdot x_{\alpha_4} &= x_{\alpha_4} - \lambda z_2 + \lambda^2z_3,\\
    u_\lambda v_\lambda \cdot z_1 &= z_1+\lambda z_2,\\
    u_\lambda v_\lambda \cdot z_2 &= z_2,\\
    u_\lambda v_\lambda \cdot z_3 &= z_3,
  \end{align*}
  while $[v_1,u_\mu]$ acts as the identity on $\fraku_{\le 3}'$. Hence
  $\fraku_{\le 3}'$ decomposes as $U'$-module into the direct sum of
  the two vector spaces
  \begin{align*}
    \fraku_{\le 3}^3 &:= \langle z_1, z_2\rangle,\\
    \fraku_{\le 3}^4 &:= \langle x_{\alpha_4}+z_1, z_3\rangle.
  \end{align*}
  The vector bundle $\V^3$ of rank $2$ over $C$ associated with the
  $U'$-torsor $\U'$ via $\fraku_{\le 3}^3$ is isomorphic to $\F_2$
  because $\fraku_{\le 3}^3$ is isomorphic to the two-dimensional
  standard representation of $U'$. The vector bundle $\V^4$ of rank
  $2$ over $C$ associated with the $U'$-torsor $\U'$ via $\fraku_{\le
    3}^4$ is isomorphic to the Frobenius pullback of $\F_2$. Indeed,
  $\fraku_{\le 3}^4$ is isomorphic to the Frobenius pullback of the
  two-dimensional standard representation of $U'$ since $u_\lambda
  v_\lambda$ acts on $\fraku_{\le 3}^4$ as the matrix 
  $\left(\begin{smallmatrix}1 & 0 \\ \lambda^2 & 1 \end{smallmatrix}\right)$.
\end{proof}

\begin{cor}\label{cor:D4_ad_n_C}
  For $i=0,1$, we have
  \begin{equation*}
    h^i(C, \ad_n(\B_k)_{|C}) =
    \begin{cases}
      4 &\text{if $C$ is ordinary},\\
      5 &\text{if $C$ is supersingular}.
    \end{cases}
  \end{equation*}
\end{cor}

\begin{proof}
  If $C$ is ordinary, then the endomorphism of $H^1(C,\OO_C)$ induced by
  Frobenius is nonzero, and therefore, $\Frob^*\F_2 \cong \F_2$. Otherwise,
  $C$ is supersingular, so $\Frob^*\F_2 \cong \OO_C^2$.  Using
  Proposition~\ref{prop:D4_ad_C} and $h^i(C,\F_r) = 1$ for $i=0,1$ and all
  $r \ge 1$, we conclude that
  \begin{equation}\label{eq:D4_ad_n_C_3}
    h^i(C, \ad_n(\B_{\le 3,k})_{|C}) =
    \begin{cases}
      4 &\text{if $C$ is ordinary},\\
      5 &\text{if $C$ is supersingular}.
    \end{cases}
  \end{equation}
  
  Now consider the long exact cohomology sequence associated with
  (\ref{eq:ad_sequence}) for $n \ge 4$. Its connecting homomorphism
  (\ref{eq:ad_sequence_connecting}) is surjective since
  Lemma~\ref{lem:cup_epi} is again valid for $n \ge 4$ in the
  $\mathbf{D}_4$-case in characteristic $2$. Therefore,
  \begin{equation*}
    h^1(C,\ad_n(\B_{\le n,k})_{|C}) = h^1(C,\ad_n(\B_{\le n-1,k})_{|C}).
  \end{equation*}
  By induction starting with (\ref{eq:D4_ad_n_C_3}), we
  obtain the result for $i=1$. The result for $i=0$ follows since the
  Euler characteristic of $\ad_n(\B_k)_{|C}$ vanishes.
\end{proof}

\begin{proof}[Proof of Proposition~\ref{prop:D4_ad}]
  By Corollary~\ref{cor:D4_ad_n_C}, Lemma~\ref{lem:S_to_C}, and the
  discussion of elliptic curves on the two isomorphism classes of
  singular cubic surfaces, we have
  \begin{equation*}
    h^i(\tS_k, \ad_n(\B_k)) =
    \begin{cases}
      4 &\text{for $S_k$ as in~(\ref{eq:D4_cubic_ordinary})},\\
      5 &\text{for $S_k$ as in~(\ref{eq:D4_cubic_supersingular})}
    \end{cases}
  \end{equation*}
  for $i=0,1$. By (\ref{eq:ad_n_to_ad}) for sufficiently large $n$, it
  suffices to prove that the connecting homomorphism
  \begin{equation}\label{eq:D4_connecting}
    H^0(\tS_k, \Lambda^\vee \otimes_\ZZ \OO_{\tS_k}) \to H^1(\tS_k, \ad_n(\B_k))
  \end{equation}
  has rank $4$. Indeed, this rank is at least $4$ since the
  composition of~(\ref{eq:D4_connecting}) with the natural map
  \begin{equation*}
    H^1(\tS_k,\ad_n(\B_k)) \to H^1(\tS_k,\ad_n(\B_{\le 1,k})) \cong 
    \bigoplus_{\alpha \in \Delta} H^1(\tS_k, L_{\alpha,k}) \cong \bigoplus_{\alpha \in \Delta} k \cong k^4
  \end{equation*}
  is the surjective connecting homomorphism $\delta$
  from~(\ref{eq:H^1(ad_B1)}). On the other hand, the rank
  of~(\ref{eq:D4_connecting}) is at most $4$ since this map factors
  through the projection
  \begin{equation*}
    H^0(\tS_k, \Lambda^\vee \otimes_\ZZ \OO_{\tS_k}) \to 
    \bigoplus_{\alpha\in\Delta} H^0(\tS_k, \OO_{\tS_k}) \cong 
    \bigoplus_{\alpha \in \Delta} k \cong k^4
  \end{equation*}
  given by the simple roots $\alpha$.
\end{proof}

\begin{cor}\label{cor:d4_ad_B_G_G'}
  The $R$-modules $H^1(\tS, \ad(\B))$, $H^1(\tS, \ad(\G))$, $H^1(S, \ad(\G'))$ are all
  zero if $\ell_1 \cap \ell_2 \cap \ell_3 = \emptyset$ on $S_k$, and are all
  nonzero otherwise.
\end{cor}

\begin{proof}
  Since $H^2(\tS_k, \ad(\B_k))=0$ because of
  (\ref{eq:Hi_Lalpha_over_k}), Cohomology and Base Change
  \cite[Theorem~III.12.11]{MR0463157} implies that the natural map
  \begin{equation*}
    H^1(\tS, \ad(\B)) \otimes_R k \to H^1(\tS_k, \ad(\B_k))
  \end{equation*}
  is an isomorphism. Using Proposition~\ref{prop:D4_ad}, we conclude
  that $H^1(\tS, \ad(\B))$ vanishes if and only if $\ell_1 \cap \ell_2
  \cap \ell_3 = \emptyset$.  The isomorphisms
  (\ref{eq:isom_ad_B_ad_G}) and (\ref{eq:isom_ad_G_ad_G'}) give the
  remaining statements for $H^1(\tS, \ad(\G))$ and $H^1(S, \ad(\G'))$.
\end{proof}

\bibliographystyle{plain}

\bibliography{bundles}

\end{document}